\documentclass[12pt]{article}
\usepackage{amsmath,amsfonts,amssymb,amsthm,amscd}
\title{Galois theory, functional Lindemann-Weierstrass, and  Manin maps.}
\date{February 6,  2015}
\author{Daniel Bertrand \thanks{Thanks to the ANR-09-BLAN-0047 for support, as well as MSRI, Berkeley}\\Universit\'e P. \& M. Curie \and Anand Pillay
\thanks{Thanks to  EPSRC grant EP/I002294/1,  MSRI, Berkeley, as well as NSF grant DMS-1360702} 
\\University of Notre Dame }
\newtheorem{Theorem}{Theorem}[section]
\newtheorem{Proposition}[Theorem]{Proposition}
 
\newtheorem{Remark}[Theorem]{Remark}
\newtheorem{Lemma}[Theorem]{Lemma}
\newtheorem{Corollary}[Theorem]{Corollary}
\newtheorem{Fact}[Theorem]{Fact}

\newtheorem{Conjecture}[Theorem]{Conjecture}
   
\newcommand{\Q}{\mathbb Q}  
\newcommand{\Z}{\mathbb Z}

\newcommand{\C}{\mathbb C}

\newcommand{\Ga}{\mathbb G}

\def\mapright#1{\smash{\mathop{\longrightarrow}\limits^{#1}}}
\begin{document}
\maketitle

\begin{abstract}
We prove several new results of  Ax-Lindemann
type for semi-abelian varieties over the algebraic closure $K$ of $\C(t)$, making heavy use of  the   Galois theory of logarithmic differential equations.
Using related techniques, we also give a generalization of the theorem of the kernel for 
 abelian varieties over $K$. This paper is a continuation of \cite{LW} as well as an elaboration on the methods of Galois descent introduced in \cite{Bertrand-Fourier} and \cite{Bertrand-descent}. 
\end{abstract}

\tableofcontents

\section{Introduction}
 The paper has three  related themes,
the common feature being differential Galois theory and its applications. 

\vspace{2mm}

 Firstly, given a semiabelian variety $B$ over the algebraic closure $K$ of $\C(t)$, a $K$-rational point $a$ of the Lie algebra $LG$ of its universal vectorial extension $G = \tilde B$, and a solution $y \in G(K^{diff})$ of the logarithmic differential equation 
 $$\partial\ell n_{G}(y) = a, ~a \in LG(K),$$ 
we want to describe $tr.deg(K^{\sharp}_{G}(y)/K^{\sharp}_{G})$ in terms of ``gauge transformations" {\em over $K$}. Here  $K^{\sharp}_{G}$ is the differential field generated over $K$ by solutions in $K^{diff}$ of $\partial\ell n_{G}(-) = 0$.
 Introducing this field as base presents both advantages and difficulties. On the one hand, it allows us to use the differential Galois theory developed by the second author in  \cite{DGTI}, \cite{DGTII}, \cite{Pillay-DGTIV}, thereby replacing the study of transcendence degrees by the computation of a Galois group. On the other hand, we have only a partial knowledge of the extension $K^{\sharp}_{G}/K$. However, it was observed by the first author in \cite{Bertrand-Fourier}, \cite{Bertrand-descent} that in the case of an abelian variety, what we do know essentially suffices to perform a Galois descent from $K^{\sharp}_{G}$ to the field $K$ of the searched-for gauge transform.  In \S 2.2 and \S 3 of the present paper, we extend this principle to  semi-abelian varieties $B$ whose  toric part   is 
$\Ga_{m}$,  and give a definitive description of  $tr.deg(K^{\sharp}_{G}(y)/K^{\sharp}_{G})$ when $B$ is an abelian variety.

The main application we have in mind of these Galois theoretic results forms the second theme of our paper, and concerns Lindemann-Weierstrass statements  for the semiabelian variety $B$ over $K$, by which we mean the description  of the transcendence degree of $exp_{B}(x)$ where $x$ is a $K$-rational point of the Lie algebra $LB$ of $B$. The problem is covered  in the above setting by choosing as data
$$a := \partial_{LG}(\tilde x) \in \partial_{LG}(LG(K)),$$
where $\tilde x$ is an arbitrary $K$-rational lift of $x$ to $G = \tilde B$. This study was initiated in our joint paper \cite{LW}, where the Galois approach was mentioned, 
but only under the hypothesis that $K^{\sharp}_{G} = K$, described as {\it $K$-largeness} of $G$. There are natural conjectures in analogy with the well-known ``constant" case (where $B$ is over $\C$), although as pointed out in \cite{LW}, there are also counterexamples provided by nonconstant extensions of a constant elliptic curve by the multiplicative group. In \S 2.3 and \S  4 of the paper,  we extend the main result of \cite{LW} to the base $K^{\sharp}_{G}$, but assuming the toric part of $B$ is at most $1$-dimensional. Furthermore, we give in this case a full solution of the Lindemann-Weierstrass statement when the abelian quotient of $B$ too is $1$-dimensional. This  uses results from \cite{BMPZ} which deal with the ``logarithmic" case.   In this direction, 
we will also formulate an ``Ax-Schanuel"  type conjecture for abelian varieties over $K$. 

\vspace{2mm}

The third theme of the paper concerns the ``theorem of the kernel", which we generalize in \S 2.4 and \S 5 by proving that linear independence with respect to $End(A)$ of points  $y_{1},.., y_{n}$ in  $A(K)$ implies linear independence of $\mu_{A}(y_{1})$,..., $\mu_{A}(y_{n})$ with respect to $\C$ (this
answers a question posed to us by Hrushovski). 
 Here $A$ is an 
abelian variety over $K = \C(t)^{alg}$ with $\C$-trace $0$ and $\mu_{A}$ is the differential-algebraic  Manin map.  However, we will give an example showing that its $\C$-linear extension $\mu_{A}\otimes 1$ on $A(K)\otimes_{\Z}\C$ is not always injective. In contrast, we observe that the $\C$-linear extension $M_{K,A}\otimes 1$ of the classical (differential-arithmetic) Manin map $M_{K,A}$ is always injective.  Differential Galois theory and the logarithmic case of nonconstant Ax-Schanuel are  involved in the proofs.

\section{Statements of results }

\subsection{Preliminaries on logarithmic equations}

We will here give a quick background to the basic  notions and objects so as to be able to state our main results in the next subsections. The remaining parts 3, 4, 5  of the paper are devoted to the proofs. We refer the reader  to \cite{LW} for more details including differential algebraic preliminaries. 

\vspace{2mm}

We  fix a  differential field $(K,\partial)$   of characteristic $0$ whose field of constants $C_{K}$ is algebraically closed (and  can often be assumed to be $\C$).  We usually 
assume that  $K$ is algebraically closed, and denote by  $K^{diff}$ the differential closure of $K$.
We let $\cal U$ denote a ``universal" differential field containing $K$, with constant field $\cal C$.  If $X$ is an algebraic variety over $K$ we will identify $X$ with its set $X(\cal U)$ of $\cal U$ points, unless we say otherwise.

We start with algebraic $\partial$-groups, which provide the habitat of the (generalised) differential Galois theory of \cite{DGTI}, \cite{DGTII}, \cite{Pillay-DGTIV} discussed later on.  A (connected) {\it algebraic $\partial$-group} over $K$ is a (connected) algebraic group $G$ over $K$ together with a lifting $D$ of the derivation $\partial$ of $K$ to a derivation of the structure sheaf ${\cal O}_{G}$ which respects the group structure.  The derivation $D$ identifies with a regular homomorphic section $s$, not of the tangent bundle of $G$, but of a certain shifted tangent bundle $T_{\partial}(G)$ over $G$,  locally defined by equations $\sum_{i=1,..n}\partial P/\partial x_{i}(\bar x)u_{i} + P^{\partial}(\bar x)$, for polynomials  $P$ in the ideal of $G$, where $P^{\partial}$ is obtained by applying the derivation $\partial$ of $K$ to the coefficients of $P$. This $T_{\partial}(G)$ is itself a (connected) algebraic group over $K$. 

 We write the algebraic $\partial$-group as $(G,D)$ or $(G,s)$. 
Not every algebraic group over $K$ has a $\partial$-structure. But when $G$ is defined over the constants $C_{K}$ of $K$, there is a privileged $\partial$-structure $s_{0}$ on $G$ which is precisely the $0$-section of $TG = T_{\partial}G$.  
Given an algebraic $\partial$-group $(G,s)$ over $K$ we obtain an associated ``logarithmic derivative" $\partial \ell n_{G,s}(-)$ from $G$ to the Lie algebra $LG$ of $G$: $\partial\ell n_{G,s}(y) = \partial(y)s(y)^{-1}$, where the product is computed in the algebraic group $T_{\partial}(G)$. This is a differential rational crossed homomorphism from $G$ onto $LG$ (at the level of $\cal U$-points or points in a differentially closed field)  defined over $K$. Its kernel $Ker(\partial\ell n_{G,s})$ is a differential algebraic subgroup of $G$ which we denote $(G,s)^{\partial}$, or simply $G^\partial$ when the context is clear. Now $s$ equips  
the Lie algebra $LG$ of $G$ with its own structure of a $\partial$-group (in this case a $\partial$-module) which we call $\partial_{LG}$  (depending on $(G,s)$) and again the kernel is  denoted $(LG)^{\partial}$. 

In the case where $G$ is defined over $C_{K}$ and $s = s_{0}$, $\partial\ell n_{G,s}$  is precisely Kolchin's logarithmic derivative, taking $y\in G$ to $\partial(y)y^{-1}$. In general, as soon as $s$ is understood, we will abbreviate $\partial \ell n_{G,s}$ by $\partial\ell n_{G}$.

\vspace{2mm}

By a {\it logarithmic differential equation} over $K$ on the algebraic $\partial$-group $(G,s)$, we mean a differential equation $\partial\ell n_{G,s}(y) = a$ for some $a\in LG(K)$. 
When $G = GL_{n}$ and $s = s_{0}$  this is the equation for a fundamental system of solutions of a linear differential equation $Y' = aY$ in vector form. And more generally  for $G$ an algebraic group over $C_{K}$ and $s = s_{0}$ this is a logarithmic differential equation on $G$ over $K$ in the sense of Kolchin.  There is a well-known Galois theory here. In the given differential closure $K^{diff}$ of $K$,  any two solutions $y_{1}, y_{2}$ of $\partial\ell n_{G}(-) = a$, in $G(K^{diff})$, differ by an element in the kernel $G^\partial$ of $\partial\ell n_{G}(-)$. But $G^\partial(K^{diff})$  is precisely $G(C_{K})$. Hence $K(y_{1}) = K(y_{2})$. In particular $tr.deg(K(y)/K)$ is the same for all solutions $y$ in $K^{diff}$. Moreover $Aut(K(y)/K)$ has the structure of an algebraic subgroup of $G(C_{K})$: for any $\sigma \in Aut(K(y)/K)$, let $\rho_{\sigma}\in G(C_{K})$ be such that $\sigma(y) = y\rho_{\sigma}$. Then the map taking $\sigma$ to $\rho_{\sigma}$ is an isomorphism between $Aut(K(y)/K)$ and an algebraic subgroup $H(C_{K})$ of $G(C_{K})$, which we call the differential Galois group of $K(y)/K$.   This depends on the choice of solution $y$, but another choice yields a conjugate of $H$. Of course when $G$ is commutative, $H$ is independent of the choice of $y$. In any case $tr.deg(K(y)/K) = dim(H)$, so computing the differential Galois group gives us a transcendence estimate. 

Continuing with this Kolchin situation, we have the following well-known fact, whose proof we present
 in the  setting of the more general situation considered in Fact 2.2.(i).

\begin{Fact} [for $G/C_K$] Suppose  $K$ algebraically closed.  Then, $tr.deg(K(y)/K)$ is the dimension of a minimal connected algebraic subgroup $H$ of $G$, defined over $C_{K}$, such for some $g\in G(K)$, 
$gag^{-1} + \partial\ell n_{G}(g) \in LH(K)$. Moreover $H(C_{K})$ is the differential Galois group of $K(y)/K)$.
\end{Fact} 
\begin{proof} Let $H$ be a connected algebraic subgroup of $G$, defined over $C_{K}$ such that $H^\partial(K^{diff}) = H(C_{K})$ is the differential Galois group of $K(y)$ over $K$. Now the $H^\partial(K^{diff})$-orbit of $y$ is defined over $K$ in the differential algebraic sense, so the $H$-orbit of $y$ is defined over $K$ in the differential algebraic sense.  A result of Kolchin on constrained cohomology (see Proposition 3.2 of \cite{DGTI},
or Theorem 2.2 of \cite{Bertrand-descent}) implies that this orbit has a $K$-rational point 
$g^{-1}$. So, there exists $z^{-1}  \in H$ such that $g^{-1} = y z^{-1}$, and $z = gy$, which satisfies $K(y) = K(z)$,   is a solution of $\partial\ell n_{G}(-) = a'$ where $a' = gag^{-1} + \partial\ell n_{G}(g)$.
\end{proof}
\noindent
(Such a map $LG(K) \to LG(K)$ taking $a\in LG(K)$ to $gag^{-1}  + \partial\ell n_{G}(g)$ for some $g\in G(K)$ is called a gauge transformation.)

\vspace{2mm} 

Now in the case of an arbitrary algebraic $\partial$-group $(G,s)$ over $K$, and logarithmic differential equation $\partial\ell n_{G,s}(-) = a$ over $K$, two solutions $y_{1},y_{2}$ in $ 
G(K^{diff})$ differ by an element of $(G,s)^{\partial}(K^{diff})$ which in general may not be contained in $G(K)$.  So both to obtain a transcendence statement independent of the  choice 
of solution, as well as a Galois theory, we should work over $K^{\sharp}_{G,s}$ which is 
the (automatically differential) field generated by $K$ and $(G,s)^{\partial}
(K^{diff})$, whose algebraic closure in $K^{diff}$ will be denoted by $K^{\sharp�\; �alg}_{G,s}$. As with $\partial \ell n_G$ and $G^\partial$, we will abbreviate $K^{\sharp}_{G,s}$ as $K^\sharp_{G}$ , or even $K^{\sharp}$, when the context is clear, and similarly for its algebraic closure.  Fixing a solution $y\in G(K^{diff})$ of  $\partial\ell n_{G}(-) = a$, for $\sigma\in Aut(K^{\sharp}(y)/K^{\sharp})$, $\sigma(y) = y\rho_{\sigma}$ for unique 
$\rho_{\sigma}\in G^{\partial}(K^{diff}) = G^\partial(K^\sharp) \subseteq G(K^{\sharp})$, and again the map $\sigma \to \rho_{\sigma}$ defines an isomorphism between $Aut(K^{\sharp}(y)/K^{\sharp})$ 
and $(H,s)^{\partial}(K^{diff})$ for an algebraic $\partial$ subgroup $H$ of $(G,s)$, ostensibly defined over $K^{\sharp}$.  The $\partial$-group $H$  (or more properly $H^{\partial}$\;, or $H^{\partial}
(K^\sharp)$)  is called the {\it (differential) Galois group} of $K^{\sharp}(y)$ over $K^{\sharp}$, and when $G$ is commutative does not depend on the choice of $y$, just on the data $a \in LG(K)$ of the logarithmic equation, and in fact only on the image of $a$ in the cokernel $LG(K)/\partial \ell n_G G(K)$ of $\partial \ell n_G$.
Again $tr.deg(K^{\sharp}(y)/K^{\sharp}) = dim(H)$.
In any case, Fact 1.1 extends to this context with essentially
 the same proof.  This can also be extracted from  Proposition 3.4 of \cite{DGTI} and the set-up of  \cite{Pillay-DGTIV}. 
For the commutative case (part (ii) below) see  \cite{Bertrand-descent}, Theorem 3.2.

\begin{Fact} [for $G/K$]  Let $y$ be a solution of $\partial\ell n_{G,s}(-) = a$ in $G(K^{diff})$, and let $K^\sharp = K(G^\partial)$, with algebraic closure $K^{\sharp \; alg}$. Then 

(i)  $tr.deg(K^{\sharp}(y)/K^{\sharp})$
 is the dimension of a minimal connected algebraic $\partial$-subgroup $H$ of $G$, defined over $K^{\sharp \; alg}$ such that  
 $gag^{-1} +  \partial\ell n_{G,s}(g) \in LH(K^{\sharp \; alg})$ for some $g\in G(K^{\sharp \; alg})$. And $H^{\partial}(K^{\sharp \; alg})$
 is the differential Galois group of $K^{\sharp \; alg}(y)/K^{\sharp \; alg}$. 

(ii) Suppose that $G$ is commutative. Then, the identity component of the differential Galois group of $K^{\sharp}(y)/K^{\sharp}$ is $H^\partial(K^\sharp)$, where $H$ is the smallest algebraic $\partial$-subgroup  of   $G$ defined over $K^\sharp$ such that $a \in LH + \Q.\partial \ell n_{G,s} G(K^\sharp)$.
\end{Fact}

 We point out  that when $G$ is commutative, then in Fact 2.1 and 2.2,  the Galois group, say $H'$, of $K^\sharp(y)/K^\sharp$ is   a unique subgroup of $G$, so its identity component $H$ must indeed be  the smallest algebraic subgroup of $G$ with the required properties (see also \cite{Bertrand-descent}, \S 3.1). Of course, $H'$ is automatically connected in 2.2.(i), where the base $K^{\sharp \; alg}$ is algebraically closed, but our proofs in \S 3 will require an appeal to 2.2.(ii). Now, in this commutative case, the map $\sigma \rightarrow \rho_\sigma$ described above depends $\Z$-linearly on $a$. So, if $N = [H':H]$ denotes the number of connected components of $H'$, then replacing $a$ by $Na$ turns the Galois group into a connected algebraic group, without modifying $K^\sharp$ nor $tr.deg(K^{\sharp}(y)/K^{\sharp}) =  tr.deg(K^{\sharp}(Ny)/K^{\sharp})$. Therefore, in the computations of Galois groups later on, we will tacitly replace $y$ by $Ny$ and determine the connected component $H$ of $H'$. But it turns out that in all cases under study, we can then assume that $y$ itself lies in $H$, so the Galois group $H'$ of $K^{\sharp}(y)/K^{\sharp}$  coincides with $H$ and will in the end always be connected   
\footnote{~We take opportunity of this paper to mention two errata in \cite{Bertrand-descent} : in the proof of its Theorem 3.2, 
replace ``of finite index" by ``with quotient of finite exponent";  in the proof of Theorem 4.4, use the reduction process described above to justify that the Galois group is indeed connected.}.

\subsection{Galois theoretic results}

The question which we deal with in this paper  is when and whether in Fact 2.2, it suffices to consider $H$ defined over $K$ and $g\in G(K)$. In fact it is not hard to see that the Galois group is defined over $K$, but the second point is problematic. The case where $(G,s)$ is a $\partial$-module, namely $G$ is a vector space $V$, and the logarithmic derivative $\partial\ell n_{G,s}(y)$ has the form $\nabla_{V}(y) = \partial y - By$ for some $n\times n$ matrix $B$ over $K$, was considered in  \cite{Bertrand-unipotent}, and shown to provide counterexamples, unless   the $\partial$-module $(V,\nabla_{V})$ is semisimple. The rough idea is that the Galois group $Gal(K^\sharp_V/K)$ of $\nabla_V$  is then reductive, allowing an argument of  Galois descent from $K^\sharp_V$ to $K$ to  
construct a $K$-rational gauge transformation $g$. The argument was extended in  \cite{Bertrand-Fourier}, \cite{Bertrand-descent} to $\partial$-groups $(G,s)$ attached to abelian varieties, which by Poincar\'e reducibility, are in a sense again semi-simple.

\smallskip 

We will here focus on  the {\it almost semiabelian} case: namely certain $\partial$-groups attached to semiabelian varieties, which provide the main source of non semi-simple situations. 
If $B$ is a semiabelian variety over $K$, then $\tilde B$, the universal vectorial extension of $B$, is a (commutative) algebraic group over $K$ which has a {\em unique} algebraic $\partial$-group structure.  Let $U$ be any unipotent algebraic $\partial$-subgroup of $\tilde B$. Then $\tilde B/U$ with its unique $\partial$-group structure is what we mean by an almost semiabelian $\partial$-group over $K$. When $B$ is an abelian variety $A$ we call $\tilde A/U$ an almost abelian algebraic $\partial$-group over $K$.  
If $G$ is an almost semiabelian algebraic $\partial$-group over $K$, then because the $\partial$-group structure $s$ on $G$ is unique, the abbreviation $K^{\sharp}_{G}$  for  $K^{\sharp}_{G,s}$ is  now unambiguous. 
We found no obstruction for the following to be true, where for different reasons we take $K$ to be $\C(t)^{alg}$  (in fact the algebraic closure of a function field in one variable over the constants is enough).

\begin{Conjecture}  Let $G$ be an almost semiabelian $\partial$-group over $K = \C(t)^{alg}$. Let $a\in LG(K)$, and $y\in G(K^{diff})$ be such that $\partial\ell n_{G}(y) = a$.
Then $tr.deg(K^{\sharp}_{G}(y)/K^{\sharp}_{G})$ is the dimension of the smallest algebraic $\partial$-subgroup $H$ of $G$ defined over $K$ such that $a\in LH + \partial\ell n_{G}(G(K))$,  i.e. $a + \partial\ell n_{G}(g)\in LH(K)$ for some $g\in G(K)$ . Equivalently the smallest algebraic $\partial$-subgroup $H$ of $G$, defined over $K$, such that $y\in H + G(K) + G^{\partial}(K^{diff})$. Moreover $H^{\partial}(K^{diff})$ is  
the Galois group of $K^{\sharp}_{G}(y)$ over $K^{\sharp}_{G}$. 
\end{Conjecture}

The conjecture can be restated as: there is a smallest algebraic $\partial$-subgroup $H$  of $(G,s)$ defined over $K$ such that $a\in LH + \partial\ell n_{G}(G(K))$ and it coincides with the Galois group of $K^{\sharp}_{G}(y)$ over $K^{\sharp}_{G}$. In comparison with Fact 2.2.(ii), notice that since $K$ is algebraically closed, $\partial \ell n_G(G(K))$ is already a $\Q$-vector space, so we do not need to tensor with $\Q$ in the condition on $a$.

A corollary of Conjecture 2.3 is the following  special {\it generic  case}, where an additional assumption on non-degeneracy is made on $a$:

\begin{Conjecture}  Let $G$ be an almost semiabelian $\partial$-group over $K = \C(t)^{alg}$. Let $a\in LG(K)$, and $y\in G(K^{diff})$ be such that $\partial\ell n_{G}(y) = a$.
Assume that $a\notin LH+ \partial\ell n_{G} G(K)$ for any proper algebraic $\partial$-subgroup $H$ of $G$, defined over $K$ (equivalently $y\notin H + G(K) + G^{\partial}(K^{diff})$ for any proper algebraic $\partial$-subgroup of $G$ defined over $K$). Then $tr.deg(K^{\sharp}_{G}(y)/K^{\sharp}_{G}) = dim(G)$. 
\end{Conjecture}

We will prove the following results in the direction of  Conjectures 2.3 and (the weaker) 2.4. 

\begin{Proposition} Conjecture 2.3  holds when $G$ is ``almost abelian". 
\end{Proposition}
\noindent 
The truth of the weaker Conjecture 2.4 in the almost abelian case   is already established in  \cite{Bertrand-Fourier}, Section 8.1(i). This reference does not address Conjecture 2.3 itself,  even if  in this case, the ingredients for its proof are there (see also \cite{Bertrand-descent}). So we take the liberty to give a reasonably self-contained  proof of Proposition 2.5 in Section 3.

\vspace{2mm}
 
As announced above, one of the  main points of the Galois-theoretic part of this paper is to try to extend Proposition 2.5 to the almost semiabelian case. 
Due to technical complications, which will be discussed later, we restrict our attention to the simplest possible extension of the almost abelian case, namely where the toric part of the semiabelian variety is $1$-dimensional, and also we sometimes just consider the  generic  case. So   the next proposition gives Conjecture 2.4 under the restriction on the toric part.
 For simplicity of notation we will work with an almost semiabelian $G$ of the form $\tilde B$ for $B$ semiabelian. 

\begin{Theorem}  Suppose that $B$ is a semiabelian variety over $K = \C(t)^{alg}$ with toric part of dimension $\leq 1$. Let $G = \tilde B$, $a\in LG(K)$ and $y\in G(K^{diff})$ a solution of $\partial\ell n_{G}(-) = a$.  Suppose that for no proper algebraic $\partial$-subgroup $H$ of $G$ defined over $K$  is 
$y\in H + G(K)$. Then $tr.deg(K^{\sharp}_{G}(y)/K^{\sharp}_{G}) = dim(G)$ and $G^{\partial}(K^{diff})$ is the differential Galois group.
\end{Theorem}

Note that the hypothesis  above `` $y\notin H + G(K)$ for any proper algebraic $\partial$-subgroup of $G$ over $K$ " is formally weaker than `` $y\notin H + G(K) + G^{\partial}(K^{diff})$ for any proper algebraic $\partial$-subgroup of $G$ over $K$ "  but nevertheless suffices, as shown by the proof of 2.6  in  Section 3.2. 
More specifically, assume that $G = \tilde A$ for a simple traceless abelian variety $A$,  that the maximal unipotent $\partial$-subgroup $U_A$ of $\tilde A$ vanishes, and that $a = 0 \in L\tilde A(K)$. The proposition then {\it implies} that any $y \in \tilde A^\partial(K^{diff})$ is actually defined over $K$, so $K^\sharp_{\tilde A} = K$. As in \cite{Bertrand-Fourier}, \cite{Bertrand-descent}, this property of  
$K$-largeness of $\tilde A$ (when $U_A = 0$) is in fact one of the main {\it ingredients} in the proof of 2.6.
As explained in \cite{DGTIII} it is based on the {\em strong minimality} of ${\tilde A}^{\partial}$  in this context, which was itself proved in an unpublished paper \cite{HS}, although there are other published accounts. Recently it has been noted in \cite{BBP-ML} that this $K$-largeness property can be seen rather more directly, using only {\em simplicity} of $A$, avoiding in particular ``deep" results from either model theory or differential algebraic geometry.

\vspace{2mm}

Our last Galois-theoretic result requires  the {\it semiconstant} notions introduced in \cite{LW}, although our notation will be a slight modification of that in \cite{LW}.  First a connected algebraic group $G$ over $K$ is said to be constant if $G$ is isomorphic (as an algebraic group) to an algebraic group defined over $\C$ (equivalent $G$ arises via base change from an algebraic group $G_{\C}$ over $\C$). For $G$ an algebraic group over $K$, $G_{0}$ will denote the largest (connected) constant algebraic subgroup of $G$. 
We will concentrate on the case $G = B$ a semiabelian variety over $K$, with $0\to T \to B \to A\to 0$ the canonical exact sequence.  So now $A_{0}$, $B_{0}$ denote the constant parts of $A,B$ respectively.  The inverse image of $A_{0}$ in $B$ will be called the semiconstant part of $B$ and will now be denoted by  $B_{sc}$. We will call $B$ semiconstant if $B = B_{sc}$ which is equivalent to requiring that $A = A_{0}$, and  moreover allows the possibility that $B = B_{0}$ is constant. 
(Of course, when $B$ is constant, $\tilde B$, which is also constant, obviously satisfies Conjecture 2.3, in view of Fact 2.1.)

\begin{Theorem} Suppose that $K = \C(t)^{alg}$ and that $B = B_{sc}$ is a semiconstant semiabelian variety over $K$ with toric part of dimension $\leq 1$.  Then Conjecture 2.3 holds for $G = \tilde B$. 
\end{Theorem}

\subsection{Lindemann-Weierstrass via Galois theory}

We are now ready to describe the impact of the previous Galois theoretic results on Ax-Lindemann problems, where $a = \partial_{LG} (\tilde x) \in \partial_{LG}(LG(K))$.  

\vspace{2mm}

Firstly, from Theorem 2.6  we will deduce directly the main result of \cite{LW} (Theorem 1.4), when $B$ is semiabelian with toric part at most  $\Ga_{m}$,  but now with transcendence degree computed over $K^{\sharp}_{\tilde B}$:

\begin{Corollary} Let $B$ be a semiabelian variety over  $K 
= \C(t)^{alg}$ 
such that the toric part of $B$ is of dimension $\leq 1$  and $B_{sc} = B_{0}$ (i.e. the semiconstant part $B_{sc}$ of $B$ is constant). Let $x\in LB(K)$, and lift $x$ to $\tilde x\in L\tilde B (K)$.  Assume that 
\newline
$(^*)$ for no proper algebraic subgroup $H$ of $\tilde B$ defined over $K$ is $\tilde x\in LH(K) +  (L{\tilde B})^{\partial}(K)$, 
\newline
which under the current assumptions is equivalent to demanding that for no proper semiabelian subvariety $H$ of $B$, is $x\in LH(K) + LB_{0}(\C)$. Then,
 
 \smallskip
 
(i) any solution $\tilde y\in B({\cal U})$  of $\partial \ell n_{\tilde B}(-) = \partial_{L\tilde B}(\tilde x)$ satisfies

\smallskip

 \centerline{$tr.deg(K^{\sharp}_{\tilde B})(\tilde y)/K^{\sharp}_{\tilde B}) = dim(\tilde B)$ \;;}
 
\smallskip

(ii) in particular,   $y := exp_{B}(x)$ satisfies $tr.deg(K^{\sharp}_{\tilde B}(y)/K^{\sharp}_{\tilde B}) = dim(B)$, i.e.  is a generic point over $K^{\sharp}_{\tilde B}$ of $B$.  
\end{Corollary}

See \cite{LW} for the analytic description of $exp_{B}(x)$ in (ii) above.  
In particular $exp_{B}(x)$ can be viewed as a a point of $B({\cal U})$.  We recall briefly the argument.  Consider $B$ as the generic fibre  of a family ${\bf B}\to S$
 of complex semiabelian varieties over a complex curve $S$, and $x$ as a rational section $x: S \to L{\bf B}$ of the corresponding family of Lie algebras.  Fix a small disc $U$ in $S$, such that $x:U\to L{\bf B}$ is holomorphic, and let $exp(x) = y:U\to {\bf B}$ be the holomorphic section obtained by composing with the exponential map in the fibres.  So $y$ lives in the differential field of meromorphic functions on $U$, which contains $K$, and can thus be embedded over $K$  in the  universal differentially closed field ${\cal U}$.   So talking about  $tr.deg(K_{\tilde B}^{\sharp}(y)/K_{\tilde B}^{\sharp})$ makes sense.

Let us comment on the methods.  In \cite{LW} an essential use was made of the so-called ``socle theorem" (Theorem 4.1 of \cite{LW}) in order to prove Theorem 1.4 there.  As recalled in the introduction, 
a differential Galois theoretic approach was also mentioned (\cite{LW}, \S 6), but could be worked out only when $\tilde B$ is $K$-large. 
In the current paper, we dispose of this hypothesis, and obtain  a stronger result, namely over $K^{\sharp}_{\tilde B}$\;, but for the time being at the expense of restricting the toric part of $B$. But even using the socle theorem does not seem to allow us to drop this restriction or even simplify the proof.

 When $B=A$ is an abelian variety one obtains a stronger statement than Corollary 2.8. This is Theorem 4.4 of \cite{Bertrand-descent}, which for the sake of completeness we restate, and  will deduce  from Proposition 2.5 in  Section 4.1.

\begin{Corollary} Let $A$ be an abelian variety over $K = \C(t)^{alg}$. Let $x\in LA(K)$, and let $B$ be the smallest abelian subvariety of $A$ such that $x\in LB(K) + LA_{0}(\C)$. Let $\tilde x\in L\tilde A (K)$ be a lift of $x$ and let $\tilde y\in \tilde A({\cal U})$ be such that $\partial\ell n_{\tilde A}(\tilde y) = \partial_{L\tilde A}(\tilde x)$. Then, 
$\tilde B^\partial$ is the Galois group of  $K^{\sharp}_{\tilde A}(\tilde y)$ over $K^{\sharp}_{\tilde A}$, so

\smallskip

 (i) $tr.deg(K^{\sharp}_{\tilde A}(\tilde y)/K^{\sharp}_{\tilde A}) =  dim(\tilde B) = 2 dim(B)$, and in particular:

\smallskip

(ii) $y := exp_{A}(x)$ satisfies $tr.deg(K^{\sharp}_{\tilde A}(y)/K^{\sharp}_{\tilde A}) = dim(B)$. 
\end{Corollary}

\vspace{2mm}
 
We now return to the semiabelian context. Corollary 2.8  is not true without the assumption that the semiconstant part of $B$  is constant. The simplest possible counterexample is given in section 5.3 of \cite{LW}: $B$ is a nonconstant extension of a constant elliptic curve $E_0$ by $\Ga_{m}$, with  judicious choices of $x$ and $\tilde x$.  
Moreover $\tilde x$ will satisfy assumption (*) in Corollary 2.8, but $tr.deg(K(\tilde y)/K) \leq 1$, which is strictly smaller than $dim(\tilde B) = 3$.  We will use 2.6 and 2.7 as well as material from \cite{BMPZ} to give a full account of this situation (now over $K^{\sharp}_{\tilde B}$, of course), and more generally, for all semiabelian surfaces $B/K$, as follows:

\begin{Corollary} Let  $B$ be an extension over $K = \C(t)^{alg}$ of an elliptic curve $E/K$ by $\Ga_{m}$. Let $x\in LB(K)$ satisfy
\newline
 $(^*)$  for any proper algebraic subgroup $H$ of $B$, $x\notin LH + LB_{0}(\C)$.
\newline
 Let $\tilde x \in L\tilde B(K)$ be a lift of $x$, let $\overline x$ be its projection to $LE(K)$, and let $\tilde y \in \tilde B(\cal U)$ be such that
$\partial\ell n_{\tilde B}(\tilde y) = \tilde x$. Then,
$tr.deg(K^{\sharp}_{\tilde B}(\tilde y)/K^{\sharp}_{\tilde B}) = 3$, unless $\bar x\in LE_{0}(\C)$ in which case $tr.deg(K^{\sharp}_{\tilde B}(\tilde y)/K^{\sharp}_{\tilde B})$ is precisely $1$. 
\end{Corollary} 

\noindent
Here, $E_0$ is the constant part of $E$. Notice that in view of Hypothesis (*), $E$ must descend to $\C$ and $B$ must be non-constant (hence not isotrivial) if $x$ projects to $LE_0(\C)$.

\subsection{Manin maps}

We finally discuss the results on the Manin maps attached to abelian varieties.  The expression ``Manin map" covers at least two maps. We are here mainly concerned with the model-theoretic or {\it differential algebraic Manin map}.   
  We identify our algebraic, differential algebraic, groups with their sets of points in a universal differential field $\cal U$ (or alternatively, points in a differential closure of whatever differential field of definition we work over). 
 So for now let $K$ be a differential field, and $A$ an abelian variety over $K$.  $A$ has a smallest Zariski-dense differential algebraic (definable in ${\cal U}$) subgroup $A^{\sharp}$, which can also be described as the smallest definable subgroup of $A$ containing the torsion. The definable group $A/A^{\sharp}$ embeds definably in a commutative unipotent algebraic group (i.e. a vector group) by Buium,  and results of Cassidy on differential algebraic vector groups yield a (non canonical) differential algebraic isomorphism between $A/A^{\sharp}$ and ${\Ga}_{a}^{n}$ where $n = dim(A)$, everything being defined over $K$. 
One can ask, among other things,  why the same $n$ can be chosen. The argument, as well as precise references to works of Buium and Cassidy, appears in Fact 1.10, Fact 1.13, and Lemma 1.14 of \cite{DGTII}, and on the face of it,  uses the ordinal-valued $U$-rank from model theory.  We sketch the argument, for completeness.  So $A/A^{\sharp}$ is a differential algebraic subgroup of some $({\cal U},+)^{m}$. Cassidy's classification of such groups says that
$A/A^{\sharp}$ is isomorphic (as a differential algebraic group, so also definably in the differential field $\cal U$) to some $({\cal U},+)^{r}\times {\cal T}$ where ${\cal T}$ is a finite-dimensional differential algebraic group. Now $A$ is connected with $U$-rank $\omega^{n}$, and $A^{\sharp}$ has finite $U$-rank. So $U$-rank inequalities give that $A/A^{\sharp}$ is connected, with $U$-rank $\omega^{n}$. This implies that ${\cal T}$ is trivial, hence $r=n$.  In any case we obtain a surjective differential algebraic homomorphism from $A$ to $({\cal U, +})^{n}$, which we call the Manin homomorphism. 
 
\vspace{2mm}

There is a somewhat more intrinsic account of this Manin map. 
Let $\tilde A$ be the universal vectorial extension of $A$ as discussed above, equipped with its unique algebraic $\partial$-group structure, and  let  $W_{A}$ be the unipotent part of $\tilde A$. We have the surjective differential algebraic homomorphism $\partial\ell n_{\tilde A}: {\tilde A} \to L\tilde A$. Note that if $\tilde y\in \tilde A$ lifts $y\in A$, then the image of $\tilde A$ under $\partial\ell n_{\tilde A}$, modulo the subgroup $\partial\ell n_{\tilde A}(W_{A})$ depends only on $y$. This gives a surjective differential algebraic homomorphism from $A$ to $L\tilde A/\partial\ell n(W_{A})$ which we call $\mu_{A}$.
\begin{Remark} Any abelian variety $A/K$ satisfies:  $Ker(\mu_{A}) = A^{\sharp}$.
\end{Remark}
\begin{proof} Let $U_{A}$ be the maximal algebraic subgroup of $W_{A}$ which is a $\partial$-subgroup of $\tilde A$. Then $\tilde A/U_{A}$ has the structure of an algebraic $\partial$-group, and as explained in \cite{LW}, the canonical map $\pi:\tilde A \to A$ induces an isomomorphism between $(\tilde A/U_{A})^{\partial}$ and $A^{\sharp}$.
As (by functoriality)
 $(\tilde A)^{\partial}$ maps onto $(\tilde A/U_{A})^{\partial}$,  $\pi:\tilde A \to A$ induces a surjective map $(\tilde A)^{\partial} \to A^{\sharp}$. 
Now as the image of $\mu_{A}$ is torsion-free, $ker(\mu_{A})$ contains $A^{\sharp}$. On the other hand, if $y\in ker(\mu_{A})$ and $\tilde y \in \tilde A$  lifts  $y$, then there is $z\in W_{A}$ such that $\partial\ell n_{\tilde A}(\tilde y) = \partial\ell n_{\tilde A}(z)$. So $\partial\ell n_{\tilde A}(\tilde y - z) = 0$ and $\pi(\tilde y - z) = y$, hence $y\in A^{\sharp}$.
\end{proof}
\noindent
Hence we call $\mu_{A}$ the (differential algebraic) Manin map. The target space embeds in an algebraic vector group hence has the structure of a $\cal C$-vector space which is unique  (any definable isomorphism between two commutative unipotent differential algebraic groups is an isomorphism of $\cal C$-vector spaces). 

\vspace{2mm}

Now assume that $K = \C(t)^{alg}$ and that $A$ is an abelian variety over $K$ with $\C$-trace $A_0 = 0$.
Then the ``model-theoretic/differential algebraic theorem of the kernel" is  (see Corollary K3 of \cite{LW}):
\begin{Fact} [$K = \C(t)^{alg}, A/K$   traceless] $Ker(\mu_{A})\cap A(K)$ is precisely the subgroup $Tor(A)$ of torsion points of $A$.
\end{Fact}

In section 5  we generalize Fact 2.12  by proving: 

\begin{Theorem} 
[$K = \C(t)^{alg}, A/K$  traceless] 
Suppose $y_{1},..,y_{n}\in A(K)$, $a_{1},$ $..,a_{n}\in \C$ are not all $0$, and $a_{1}\mu_{A}(y_{1}) + ... + a_{n}\mu_{A}(y_{n}) = 0 \in L\tilde A(K)/\partial \ell n_{\tilde A}(W_{A})$. Then $y_{1},..,y_{n}$ are linearly dependent  over $End(A)$.
\end{Theorem}

Note that on reducing to a simple abelian variety,
 Fact 1.12 is the special case when $n=1$.   
Hrushovski asked whether the conclusion of Theorem 2.13 can be strengthened to the linear dependence  of $y_{1},..,y_{n}$ over $\Z$.  Namely is the extension $\mu_{A}\otimes 1$ of $\mu_A$ to $A(K)\otimes_{\Z}\C$  injective ? 
We found that an example of Yves Andr\'e (see \cite{LW}, p. 504, as well as \cite{LB}, IX.6) of a traceless abelian variety $A$ with  $U_{A} \neq W_{A}$ yields a counterexample. Namely:

\begin{Proposition} There exist  

- a simple traceless 4-dimensional abelian variety $A$ over $K = \C(t)^{alg}$, such that $End(A)$ is an order in a CM number-field $F$ of degree 4 over $\Q$,

-  four points $y_1,...,  y_4$ in $A(K)$ which are linearly dependent over  $End(A)$, but linearly independent over $\Z$, 

- and four complex numbers $a_1, ..., a_4$, not all zero,

\noindent
such that   $a_1 \mu_A(y_1) + ... + a_4 \mu_A(y_4) = 0$.
\end{Proposition}

\noindent
 In fact, for $i = 1, ..., 4$, we will construct lifts $\tilde y_i \in \tilde A(K)$ of the points $y_i$, and solutions $\tilde x_i \in L\tilde A(K^{diff})$ to the equations $\nabla(\tilde x_i) = \partial \ell n_{\tilde A} \tilde y_i$ (where we have set $\nabla := \nabla_{L\tilde A} =  \partial_{L\tilde A}$, with $\nabla_{|LW_A} = \partial \ell n_{\tilde A | W_A}$ in the identification $W_A = LW_A$), and will find a non-trivial relation
$$a_1 \tilde x_1 + ... + a_4 \tilde x_4 := u \in U_A(K^{diff}) \qquad ({\frak R}).$$
 Since $U_A$ is a $\nabla$-submodule of $L\tilde A$, this implies that 
$a_1 \partial \ell n_{\tilde A} \tilde y_1 + ... + a_4 \partial \ell n_{\tilde A} \tilde y_4$  lies in $U_A$.  And  since $U_A \subseteq W_A$, this in turn shows  that 
$$a_1 \mu_A(y_1) + ... + a_4 \mu_A(y_4) = 0 ~ {\rm in} ~ L\tilde A / \partial \ell n_{\tilde A}(W_A),$$ 
contradicting the injectivity of $\mu_A \otimes 1$. 

\vspace{2mm}

We conclude with a remark on the more classical  {\it differential  arithmetic}  Manin map $M_{K,A}$\;, where the stronger version {\it is} true.
Again $A$ is an abelian variety over $K = \C(t)^{alg}$ with $\C$-trace $0$. As above, we let $\nabla$ denote $\partial_{L\tilde A}:L\tilde A \to L\tilde A$.  The map $M_{K,A}$ is then the
homomorphism from $A(K)$ to   $L\tilde A(K)/\nabla(L\tilde A(K))$ which attaches to a point $y \in A(K)$ the class  $M_{K,A}(y)$ of $\partial\ell n_{\tilde A}(\tilde y)$ in  $L\tilde A(K)/\nabla(L\tilde A(K))$, for any 
lift $\tilde y$ of $y$ to $\tilde A(K)$. This class is independent of the lift, since $\partial \ell n_{\tilde A}$ and $\partial_{L\tilde A}$ coincide on $W_A = LW_A$. Again $L\tilde A(K)/\nabla(L\tilde A(K))$ is a $\C$-vector space. The initial theorem of Manin (see \cite{Coleman})
says that $Ker(M_{K,A}) = Tor(A) + A_{0}(\C)$, so in the traceless case is precisely $Tor(A)$.

\begin{Proposition} 
[$K = \C(t)^{alg}, A/K$   traceless] 
 The $\C$-linear extension
 $M_{K,A}\otimes 1: A(K)\otimes_{Z}\C \to L\tilde A(K)/\nabla(L\tilde A(K))$ is injective. 
\end{Proposition}

\section{Computation of Galois groups}

Here we prove the Galois theoretic statements   2.5, 2.6  and 2.7 announced in \S 2.2. We assume throughout that  $K = \C(t)^{alg}$.

\subsection{The abelian case}

Let us first set up the notations.  Let $A$ be an abelian variety over $K$, and let $A_{0}$ be its $\C$-trace, which we view as a subgroup of $A$ defined over $\C$. Let $\tilde A$ be the universal vectorial extension of $A$. We have the short exact sequence
$0\to W_{A}\to \tilde A \to A \to 0$. 
Let $U_{A}$ denote the (unique) maximal $\partial$-subgroup of $\tilde A$ contained in $W_{A}$.  By Remark 7.2 of \cite{Bertrand-Fourier}, we have:

\begin{Fact} ${\tilde A}^{\partial}(K^{diff}) = \tilde A_{0}(\C) + Tor(\tilde A) + U_{A}^{\partial}(K^{diff})$. 
\end{Fact}

Let us briefly remark that the ingredients behind Fact 3.1 include Chai's theorem (see \cite{Chai} and  \S K of \cite{LW}), as well as the strong minimality of $A^{\sharp}$ when $A$ is simple and traceless from \cite{HS}. As already pointed out in connection with $K$-largeness,  the reference to \cite{HS} can be replaced by the easier arguments from \cite{BBP-ML}. 
Let $K_{\tilde A}^{\sharp}$ be the (automatically differential) field  generated over $K$ by ${\tilde A}^{\partial}(K^{diff})$, and likewise with  $K_{U_{A}}^{\sharp}$ for $(U_{A})^{\partial}(K^{diff})$.  
So by Fact 3.1 $K_{\tilde A}^{\sharp} = K_{U_{A}}^{\sharp}$.  Also, as recalled at the beginning of Section 8 of \cite{Bertrand-Fourier}, we have:
\begin{Remark} $K_{U_{A}}^{\sharp}$ is a Picard-Vessiot extension of $K$ whose Galois group (a linear algebraic group over $\C$) is semisimple. 
\end{Remark} 

\subsubsection*{Proof of Proposition 2.5}
Here, $G$ is  an {\em almost abelian} $\partial$-group over $K$, and for simplicity of notation
we assume $G = \tilde A$, as was announced for the semi-abelian case. 
Fix a copy $K^{diff}$ of the differential closure of $K$, and let $y\in G(K^{diff})$ be such that $a = \partial \ell n_{G}(y)$ lies in $LG(K)$.  Note that in the set-up of Conjecture 2.3, $y$ could be very well be an element of $U_{A}$, for instance when $a\in LU_{A} = U_{A}$, so in a sense we move outside the almost abelian context. 
In any case, let $H$ be a minimal connected $\partial$-subgroup of $G$ defined over $K$ such that $y\in H + G(K) + G^{\partial}(K^{diff})$.  
We will prove that $H^{\partial}(K^{diff})$ is the  
differential Galois group of $K^{\sharp}(y)$ over $K^{\sharp}$ where $K^{\sharp} = K^{\sharp}_{G}$. 
We recall from the comments after Fact 2.2 on the commutative case that we can and do assume that this Galois group is connected. Also, the statement   implies that $H$ is actually the smallest connected $\partial$-subgroup of $G$ over $K$ such that  $y\in H + G(K) + G^{\partial}(K^{diff})$, as required.

Let $H_{1}^{\partial}$ be the  
Galois group of $K^{\sharp}(y)$ over $K^{\sharp}$ with $H_{1}$ a $\partial$-subgroup of $G$ which on the face of it is defined over $K^{\sharp}$. Again 
by remarks after Fact 2.2, 
$H_{1}$ is a connected $\partial$-subgroup of $H$.  So we aim to show that $H = H_{1}$. 

\vspace{2mm}
\noindent
{\em Claim.} $H_{1}$ is defined over $K$ as an algebraic group.
\newline
{\em Proof.}  It is enough to show that $H_1^{\partial}$ is defined over $K$ as a differential algebraic group.  This is a very basic model-theoretic  argument, but may be a bit surprizing at the algebraic-geometric level, as $K^{\sharp}(y)$ need not be a ``differential Galois extension" of $K$ in any of the usual meanings. We use the fact that any definable (with parameters) set in the differentially closed field  $K^{diff}$ which is $Aut(K^{diff}/K)$-invariant, is definable over $K$. This follows from model-theoretic homogeneity of $K^{diff}$ over $K$ as well as elimination of imaginaries in $DCF_{0}$.  
Now $H_{1}^{\partial}(K^{diff})$ is the set of $g\in G^{\partial}(K^{diff})$ such that $y_{1}g$ and $y_{1}$ have the same type over $K^{\sharp}$ for some/any $y_{1}\in G(K^{diff})$ such that $\partial\ell n_{G}(y_{1}) = a$.  As $a\in LG(K)$ and $K^{\sharp}$ is setwise invariant under $Aut(K^{diff}/K)$ it follows that $H_{1}^{\partial}(K^{diff})$ is also $Aut(K^{diff}/K)$-invariant, so defined over $K$.  This proves the claim.

\vspace{2mm}
\noindent
Note that we may assume   $y\in H$ whereby $\partial\ell n_{G}(y) = a\in LH(K)$. 

\medskip
\noindent
Let $B$ be the image of $H$ in $A$, and $B_{1}$ the image of $H_{1}$ in $A$. So $B_{1}\leq B$ are abelian subvarieties of $A$. Let $V$ be the maximal 
unipotent $\partial$-subgroup of $H$, and $V_{1}$ the maximal unipotent subgroup of $H_{1}$. So $V_{1} \leq V$, and using the assumptions and 
the claim, everything is defined over $K$.  Note also that the surjective homomorphism $H\to B$ induces an isomorphism between $H/V$ and $\tilde 
B/U_{B}$ (where as above $U_{B}$ denotes the maximal unipotent $\partial$-subgroup of $\tilde B$). Likewise for $H_{1}/V_{1}$ and the quotient of 
$\tilde{B_{1}}$ by its maximal unipotent $\partial$-subgroup.  

\vspace{2mm}
\noindent
{\em Case (I).}  $B = B_{1}$. 
\newline
Then by the previous paragraph, we have a canonical isomorphism $\iota$ (of $\partial$-groups) between $H/H_{1}$ and $V/V_{1}$, defined over $K$, so there is no harm in identifying them, although we need to remember where they came from.  Let us denote $V/V_{1}$ by ${\overline V}$, a unipotent  $\partial$-group.  This isomorphism respects the logarithmic derivatives in the obvious sense. 
Let $\bar y$ denote the image of $y$ in $H/H_{1}$. So $\partial\ell n_{H/H_{1}}(\bar y) = \bar a$ where $\bar a$ is the image of $a$ in $L(H/H_{1})(K)$. 
Via $\iota$ we identify ${\bar y}$ with a point in ${\bar V}(K^{\sharp})$ and ${\bar a}$ with $\partial\ell n_{\bar V}({\bar y})\in L({\bar V})(K)$. 

\medskip
\noindent
By 3.2 we identify $Aut(K^{\sharp}/K)$
 with a group $J(\C)$ where $J$ is a semisimple algebraic group. We have a natural action of  $J(\C)$ on ${\bar V}^{\partial}(K^{diff}) = {\bar V}^{\partial}(K^{\sharp})$. Now the latter is a $\C$-vector space, and this action can be checked to be a (rational)  representation of $J(\C)$. On the other hand, 
for $\sigma\in J(\C)$, $\sigma({\bar y})$
 (which is well-defined since $\bar y$ is $K^{\sharp}$-rational) is also a solution of 
$\partial\ell n_{\bar V}(-) = \bar a$, hence $\sigma(\bar y) - {\bar y}\in {\bar V}^{\partial}(K^{diff})$. The map taking $\sigma$ to $\sigma({\bar y})-{\bar y}$   is then a cocycle $c$ 
from $J(\C)$ to $V^{\partial}(K^{diff})$ which is continuous for the  Zariski topologies. Now the appropriate $H^{1}(J(\C),  {\bar V}^{\partial}(K^{diff}))$ is trivial as it equals  
$Ext_{J(\C)}(1,{\bar V}^{\partial}(K^{diff}))$,  the group of isomorphism classes of extensions of the trivial representation of $J(\C)$ by  ${\bar V}^{\partial}(K^{diff})$. But $J(\C)$ is semisimple, so 
reductive, whereby every rational representation is completely reducible (see p.26 and 27 of \cite{Mumford}, and \cite{Bertrand-unipotent} for Picard-Vessiot applications, which actually cover the case when $a$ lies in $LU_A$).  Putting everything together the original cocycle  is trival. So  there is ${\bar z}\in {\bar V}^{\partial}
(K^{\sharp})$ such that $\sigma({\bar y}) - {\bar y} =  \sigma(z)-z$ for all $\sigma\in J(\C)$.  So $\sigma({\bar y} - {\bar z}) = {\bar y} - {\bar z}$ for all $\sigma$.  Hence 
${\bar y}-{\bar z}\in (H/H_{1})(K)$. Lift  ${\bar z}$ to a point $z\in H^{\partial}(K^{diff})$. So $\overline{y-z}\in {\bar V}(K)$. As $K$ is algebraically closed, there is $d\in H(K)$ such that  $y-z+d\in H_{1}$.  This contradicts   the minimal choice of $H$, unless $H = H_{1}$. So the proof is complete in Case (I).

\vspace{2mm}
\noindent
{\em Case (II).}  $B_{1}$  is a proper subgroup of $B$.
\newline
Consider the group $H_{1}.V$ a $\partial$-subgroup of $H$, defined over $K$, which also projects onto $B_{1}$.  It is now easy to extend $H_{1}.V$ to a 
$\partial$-subgroup $H_{2}$ of $H$ over $K$  such that $H/H_{2}$ is canonically isomorphic to $\overline B_{2}$, where $B_{2}$ is a simple abelian variety, and $\overline{B_{2}}$ denotes the quotient of $\tilde{B_{2}}$ by its maximal unipotent subgroup.  Now let ${\bar y}$ denote $y/H_{2}\in H/H_{2}$. 
Hence $\partial\ell n_{\overline B_{2}}({\bar y}) = {\bar a}\in L(\overline{B_{2}})(K)$.  As $H_{1}\subseteq H_{2}$, ${\bar y}\in {\overline B_{2}}(K^{\sharp})$.  Now we have two cases. If  $B_{2}$ descends to $\C$, then  ${\bar y}$ generates a strongly normal extension of $K$ with Galois group
 a connected algebraic subgroup of $B_{2}(\C)$. As this Galois group will be a homomorphic image of the linear (in fact semisimple) complex algebraic group $Aut(K^{\sharp}/K)$ we have a contradiction, unless ${\bar y}$ is $K$-rational. On the other hand, if $B_{2}$ does not descend to $\C$, then by Fact 2.2, ${\bar y}$ generates over $K$ a (generalized) differential Galois extension of $K$ with Galois group contained in ${\overline B_{2}}^{\partial}(K^{diff})$, which again will be a homomorphic image of a complex semisimple  linear algebraic group 
(cf. \cite{Bertrand-Fourier}, 8.2.i).  We get a contradiction by various possible means (for example as in Remarque 8.2 of \cite{Bertrand-Fourier}) unless  ${\bar y}$ is $K$-rational.  So either way we are forced into ${\bar y}\in (H/H_{2})(K)$. 
But then, as $K$ is algebraically closed, $y- d\in H_{2}$ for some $d\in H(K)$, again a contradiction. So Case (II) is impossible. This concludes the proof of 2.5. 

\subsection{The semiabelian case}
 
We now aim towards proofs of Theorems 2.6 and 2.7.  Here, $G$  denotes an almost semiabelian algebraic $\partial$-group over $K$.  As in the statements of 2.6 and 2.7 we make the notationally simplifying assumption that $G = \tilde B$ for $B$ a semiabelian variety over $K$, equipped with its unique algebraic $\partial$-group structure.

We have: 
\newline 
$0\to T \to B \to A\to 0$,  where $T$ is an algebraic torus and $A$ an abelian variety, all over $K$, 
\newline
$G= \tilde B = B\times_{A} \tilde A$, where $\tilde A$ is the universal vectorial extension of $A$, and
\newline
$0\to T \to G \to \tilde A \to 0$. 
\newline
We use the same notation for $A$ as at the beginning of this section,  namely
\newline
$0 \to W_{A}\to \tilde A \to A \to 0$. We denote by $A_{0}$  the $\C$-trace of $A$ (so up to isogeny we can write $A$ as a product $A_{0}\times A_{1}$, all defined over $K$, where $A_{1}$ has $\C$-trace $0$), and
by $U_{A}$   the maximal  $\partial$-subgroup of $\tilde A$ contained in $W_{A}$. So $U_A$  is a unipotent  subgroup of $G$, though not necessarily one of its $\partial$-subgroups. Finally, we  have the exact sequence:
$$0 \mapright{} T^{\partial} \mapright{} 
 G^{\partial} \mapright{\pi} {\tilde A}^{\partial} \mapright{} 0 ~.$$
Note that $T^{\partial} = T(\C)$. 
Let $K^{\sharp}_{G}$ be the (differential) field generated over $K$ by $G^{\partial}(K^{diff})$. We have already noted above that $K^{\sharp}_{\tilde A}$ equals  $K^{\sharp}_{U_{A}}$. So $K^{\sharp}_{U_{A}} < K^{\sharp}_{G}$, and we deduce from the last exact sequence above:
\begin{Remark} 
$G^{\partial}(K^{diff})$ is the union of the $\pi^{-1}(b)$ for $b\in {\tilde A}^{\partial}$, each $\pi^{-1}(b)$ being a coset of $T({\C})$ defined over $K^{\sharp}_{U_{A}}$.  Hence $K^{\sharp}_{G}$ is (generated by) a union of Picard-Vessiot extensions over $K^{\sharp}_{U_{A}}$ each with Galois group contained in $T(\C)$. 
\end{Remark}

\subsubsection*{Proof of Theorem 2.6}
Bearing in mind Proposition 2.5 we may assume that   $T = \Ga_{m}$. We have $a\in LG(K)$ and $y\in G(K^{diff})$ such that $\partial\ell n_{G}(y) = a$, and that $y \notin H + G(K)$ for any proper $\partial$-subgroup $H$ of $G$. The latter is a little weaker than the condition that $a \notin LH(K) + \partial\ell n_{G}(G(K))$ for any proper $H$, but 
(thanks to Fact 3.1) will suffice for the special case we are dealing with.
 
Fix a solution $y$ of $\partial\ell n_{G}(-) = a$ in $G(K^{diff})$ and let $H^{\partial}(K^{diff})$ be the differential Galois group
 of $K^{\sharp}_{G}(y)$ over $K^{\sharp}_G$. As said after Fact 2.2, 
 there is no harm in assuming that $H$ is connected. 
 So $H$ is a connected $\partial$-subgroup of $G$, defined over $K^{\sharp}_{G}$.  

\vspace{2mm}
\noindent
As in the proof of the claim in the proof of Proposition 2.5 we have:

\vspace{2mm}
\noindent
{\em Claim 1.} $H$ (equivalently $H^{\partial}$) is defined over $K$.

\vspace{2mm}
\noindent
We assume for a contradiction that $H\neq G$. 
\newline
{\em Case (I).} $H$ maps onto a proper ($\partial$-)subgroup of $\tilde A$. 
\newline
This is similar to the Case (II)  in the proof of Proposition 2.5 above. Some additional complications come from the structure of $K_{G}^{\sharp}$.  We willl make use of Remark 3.3 all the time.
\newline
As $\tilde A$ is an essential extension of $A$ by $W_{A}$, it follows that we can find a connected $\partial$-subgroup $H_{1}$ of $G$ containing $H$ and defined over $K$ such that the surjection $G\to \tilde A$ induces an isomorphism between $G/H_{1}$ and $\overline {A_{2}}$, where $A_{2}$ is a simple abelian subvariety of $A$ (over $K$ of course) and $\overline{A_{2}}$ is the quotient of $\tilde{A_{2}}$ by its maximal unipotent $\partial$-subgroup. 
The quotient map taking $G$ to $\overline{A_{2}}$ takes $y$ to $\eta$ say and also induces  a surjection $LG \to L(\overline{A_{2}})$ which takes $a$ to 
$\alpha \in L(\overline{A_{2}})$ say. 

\vspace{2mm}
\noindent
As $\eta = y/H_{1}$ and $H\subseteq H_{1}$, we see that $\eta$ is fixed by $Aut(K^{\sharp}_{G}(y)/K^{\sharp}_{G})$, hence

\vspace{2mm}
\noindent
{\em Claim 2.}  $\eta \in \overline{A_{2}}(K^{\sharp}_{G})$. 

\vspace{2mm}
\noindent
On the other hand $\eta$ is a solution of the logarithmic differential equation $\partial\ell n_{\overline{A_{2}}}(-) = \alpha$ over $K$, and 
using one of the ingredient in the proof of Fact 3.1,   $K^{\sharp}_{\overline{A_{2}}} = K$, hence $K(\eta)$ is a differential Galois extension of $K$ whose Galois group is either trivial (in which case $\eta \in \overline{A_{2}}(K)$), or equal to $\overline{A_{2}}^{\partial}(K^{diff})$. 

\vspace{2mm}
\noindent
{\em Claim 3.}  $\eta\in \overline{A_{2}}(K)$. 
\newline
{\em Proof.} Suppose not. We first claim that $\eta$ is independent from $K^{\sharp}_{U}$ over $K$ (in the sense of differential fields). 
Indeed, the Galois theory would otherwise give  us some proper definable subgroup in the product  
of $\overline{A_{2}}^{\partial}(K^{diff})$ by the Galois group of $K^{\sharp}_{U}$ over $K$ (or equivalently, these two groups would share a non-trivial definable quotient). As the latter is a complex semisimple algebraic group (Remark 3.2),
 we get a contradiction.  Alternatively we can proceed as in Remarque 8.2 of \cite{Bertrand-Fourier}. 
\newline
So the Galois group of $K^{\sharp}_{U}(\eta)$ over $K^{\sharp}_{U}$ is  $\overline{A_{2}}^{\partial}(K^{diff})$. 
 As there are no nontrivial  
 definable subgroups of $\overline{A_{2}}(K^{diff}) \times \Ga_{m}(\C)^{n}$,  we see that $\eta$ is independent of $K^{\sharp}_{G}$ over $K^{\sharp}_{U}$ contradicting Claim 2. 

\vspace{2mm}
\noindent
By Claim 3, the coset of $y$ modulo $H_{1}$ is defined over $K$ (differential algebraically), so as in the proof of Fact 2.1,  as $K$ is algebraically closed there is $y_{1}\in G(K)$ in the same coset of 
$H_1$ as $y$. So $y\in H_1 + G(K)$, contradicting the assumptions. So Case (I) is complete.

\vspace{2mm}
\noindent
{\em Case (II).} $H$ projects on to $\tilde A$. 
\newline
Our assumption that $H$ is a proper subgroup of $G$ and that the toric part is $\Ga_{m}$ implies  that (up to isogeny) $G$ splits as $T\times H  = T\times \tilde A$.  The case is essentially dealt with in \cite{Bertrand-Fourier}. But nevertheless we continue with the proof. 
We identify $G/H$ with $T$. So $y/H = d \in T$ and the image $a_{0}$ of $a$ under the projection $G \to T$ is in $LT(K)$. As $H^{\partial}(K^{diff})$ 
is the Galois group of $K^{\sharp}_{G}(y)$ over $K^{\sharp}_{G}$, we see that $y\in T(K^{\sharp}_{G})$. Now $K(d)$ is a Picard-Vessiot extension of 
$K$ with Galois group a subgroup of $\Ga_{m}(\C)$. Moreover as $G$ splits as $T\times \tilde A$, $G^{\partial} = T^{\partial} \times {\tilde 
A}^{\partial}$. Hence by Fact 3.1, $K^{\sharp}_{G} = K^{\sharp}_{\tilde A}$ and by Remark 3.2, it is a Picard-Vessiot extension of $K$ with Galois group 
a semisimple 
algebraic group in the constants. We deduce from the Galois theory that  $d$ is independent from $K^{\sharp}_{G}$ over $K$, hence $d\in T(K)$. 
So the coset of $y$ modulo $H$ has a representative $y_{1}\in G(K)$ and $y \in H + G(K)$, contradicting our assumption. This concludes Case (II)  and the proof of Theorem 2.6.

\vspace{5mm}
\noindent
 
\subsubsection*{Proof of Theorem 2.7}

So  $G = \tilde B$ for $B = B_{sc}$ a semiconstant semiabelian variety over $K$ and we may assume it has toric  part $\Ga_{m}$.  So although the toric part is still $\Ga_{m}$ both the hypothesis and conclusion of 2.7 are stronger than in 2.6. 

We have $0\to \Ga_{m} \to B \to A$ where $A = A_{0}$ is over $\C$, hence also $\tilde A$ is over $\C$ and we have $0\to \Ga_{m} \to \tilde B \to \tilde A \to 0$, and $G = \tilde B$.  As ${\tilde A}^{\partial} = {\tilde A}(\C) \subseteq \tilde A(K)$, we see that
\begin{Fact} $G^{\partial}(K^{diff})$ is a union of cosets of $\Ga_{m}(\C)$, each defined over $K$. 
\end{Fact}

We are given a logarithmic differential equation $\partial\ell n_{G}(-) = a\in LG(K)$ and solution $y\in G(K^{diff})$. We let $H$ be a minimal connected
$\partial$-subgroup of $G$, defined over $K$, such that $a\in LH + \partial\ell n_{G}(G(K))$, equivalently $y\in H + G(K) + G^{\partial}(K^{diff})$. 
We want to prove that $H^{\partial}(K^{diff})$ is  
 the Galois group of $K^{\sharp}_{G}(y)$ over $K^{\sharp}_{G}$. 
\newline
By Theorem 2.6 we may assume that $H \neq G$. Note that after translating $y$ by an element of $G(K)$ plus an element of $G^{\partial}(K^{diff})$ we can assume that $y\in H$.  
If $H$ is trivial then everything is clear.

\vspace{2mm}
\noindent
We go through the cases.
\newline
{\em Case (I)}. $H = \Ga_{m}$.  
\newline
Then by Fact 2.1,  $K(y)$  is a Picard-Vessiot extension of $K$, with Galois group $\Ga_{m}(\C)$, and all that remains to be proved is that $y$ is algebraically independent from $K^{\sharp}$ over $K$.  Let $z_{1},..,z_{n}\in G^{\partial}(K^{diff})$, and we want 
to show that $y$ is independent from $z_{1},..,z_{n}$ over $K$ (in the sense of $DCF_{0}$). By Fact 3.4, $K(z_{1},..,z_{n})$ is a Picard-Vessiot extension of $K$ and we can assume the Galois group is $\Ga_{m}^{n}(\C)$.  Suppose towards a contradiction that  
$tr.deg(K(y,z_{1},..,z_{n}/K) < n+1$ so has to equal $n$.   Hence the differential Galois group of $K(y,z_{1},..,z_{n})/K)$ is of the form $L(\C)$ where $L$ is the algebraic subgroup of $\Ga_{m}^{n+1}$ defined by $kx + k_{1}x_{1} + ... + k_{n}x_{n} = 0$ for $k, k_{i}$ integers, $k\neq 0$,  not all $k_{i}$ zero. It easily follows that
$ky + k_{1}z_{1} + .. k_{n}z_{n}\in G(K)$. So $ky$ is of the form $z+g$ for $z\in G^{\partial}(K^{diff})$ and $g\in G(K)$. Let $z'\in G^{\partial}$ and $g'\in G(K)$ be such that 
$kz' = z$ and $kg' = g$. Then $k(y-  (z'+g')) = 0$, so $y - (z'+g)$ is a torsion point of $G$ hence also in $G^{\partial}$. We conclude that $y\in G^{\partial}(K^{diff}) + G(K)$, contradicting our assumptions on $y$. This concludes the proof in Case (I).

\vspace{2mm}
\noindent
{\em Case (II).} $H$ projects onto $\tilde A$.
\newline
So our assumption that $G \neq H$ implies that up to isogeny $G$ is $T\times \tilde A$ so defined over $\C$, and everything follows from Fact 2.1.

\vspace{2mm}
\noindent
{\em Case (III).}  Otherwise. 
\newline
This is more or less a combination of the previous cases. 
\newline
To begin, suppose $H$ is disjoint from $T$ (up to finite).  So $H \leq \tilde A$ is a constant group, and by Fact 2.1, $H^{\partial}(K^{diff}) = H(\C)$ is the Galois group of $K(y)$ over $K$. By Fact 3.4 the Galois theory tells us that $y$ is independent from $K^{\sharp}_{G}$ over $K$, so $H(\C)$ is the Galois group of $K^{\sharp}(y)$ over $K^{\sharp}$ as required.
\newline
So we may assume that $T\leq H$. Let $H_{1} \leq H$ be the differential Galois group of $K^{\sharp}_{G}(y)$ over $K^{\sharp}_{G}$, and we suppose for a contradiction that $H_{1} \neq H$. As in the proof of 2.5, $H_{1}$ is defined over $K$. 
By the remark after Fact 2.2, we can assume  that  $H_1$ is connected.

\vspace{2mm}
\noindent
{\em Case (III)(a).} $H_{1}$ is a complement of $T$ in $H$ (in the usual sense that $H_{1}\times T \to H$ is an isogeny). 
\newline
So $y/H_{1} \in T(K^{\sharp}_{G})$. Let $y_{1} = y/H_{1}$.  If $y_{1}\notin T(K)$, $K(y_{1})$ is a Picard-Vessiot extension of $K$ with Galois group $\Ga_{m}(\C)$. The proof in Case (I) above shows that $y_{1}\in G^{\partial}(K^{diff}) + G(K)$ whereby $y\in H_{1} + G^{\partial}(K^{diff}) + G(K)$, contradicting the minimality assumptions on $H$. 

\vspace{2mm}
\noindent
{\em Case (III)(b).}  $H_{1} + T$ is a proper subgroup of $H$. 
\newline
Note that as we are assuming $H_{1}\neq H$, then the negation of Case (III)(a) forces Case (III)(b) to hold.  Let $H_{2} = H_{1} + T$, so $H/H_{2}$
 is a constant group $H_{3}$ say which is a vectorial extension of an abelian variety. Then  $y_{2} = y/H_{2} \in H_{3}(K^{\sharp}_{G})$, and $K(y_{2})$ is a Picard-Vessiot extension of $K$ with Galois group a subgroup of $H_{3}(\C)$.  Fact 3.4 and the Galois theory implies that 
 $y_{2}\in H_{3}(K)$.  Hence  $y \in H_{2} + G(K)$, contradicting the minimality of $H$ again.
\newline 
This completes the proof of Theorem 2.7. 

\subsection{Discussion on non generic cases}
We complete this section with a discussion of some complications arising when one would like to drop either the genericity assumption in Theorem 2.6, or the restriction on the toric part in both Theorems 2.6 and 2.7. 

 \vspace{2mm}
Let us first give an example which will have to be considered  if we drop the genericity assumption in 2.6, and give some positive information as well as identifying some technical  complications. Let $A$ be a simple abelian variety over $K$ which has $\C$-trace $0$ and such that $U_{A}\neq 0$. (Note that such an example appears below in Section 5.2 connected with Manin map issues.) Let $B$ be a nonsplit extension of $A$ by $\Ga_m$, and let $G =\tilde B$.  We have $\pi:G\to {\tilde A}$ with kernel $\Ga_{m}$, and let $H$ be $\pi^{-1}(U_{A})$, a $\partial$-subgroup of $G$.  Let $a\in LH(K)$ and $y\in H(K^{diff})$ with $d\ell n_{H}(y) = a$. We have to compute $tr.deg(K^{\sharp}_{G}(y)/K^{\sharp}_{G})$. Conjecture 2.3 predicts that it is the dimension of the smallest algebraic $\partial$-subgroup $H_{1}$ of $H$ such that $y\in H_1 
+ G(K) + G^{\partial}(K^{diff})$. 

\begin{Lemma} With the above notation: Suppose that $y\notin H_{1} + G(K) + G^{\partial}(K^{diff})$ for any proper algebraic $\partial$-subgroup $H_1$ of $H$ over $K$.  Then $tr.deg(K^{\sharp}_{G}(y)/K^{\sharp}_{G}) = dim(H)$ (and $H$ is the Galois group). 
\end{Lemma}
\begin{proof} Let $z$ and $\alpha$ 
the images of $y$, $a$ respectively under the maps $H\to U_{A}$ and $LH\to L(U_{A}) = U_{A}$ induced by $\pi:G\to \tilde A$.  So $\partial\ell n_{\tilde A}(z) = \alpha$ with $\alpha \in L \tilde A(K)$. 

\vspace{2mm}
\noindent
{\em Claim.}  $z\notin   U + {\tilde A}(K) + {\tilde A}^{\partial}(K^{diff})$ for any proper algebraic $\partial$-subgroup $U$ of $U_{A}$, over $K$. 
\newline
{\em Proof of claim.}  Suppose otherwise. Then lifting suitable $z_{2}\in {\tilde A}(K)$, $z_{3}\in {\tilde A}(K^{diff})$, to $y_{2}\in G(K)$, $y_{3}\in G^{\partial}(K^{diff})$ respectively, we see that $y - (y_{2} + y_{3}) \in \pi^{-1}(U)$, a proper algebraic $\partial$-subgroup of $H$, a contradiction. 

\vspace{2mm}
\noindent

As in  Case (I) in the proof of   2.5, we conclude that $tr.deg(K^{\sharp}_{\tilde A}(z)/K^{\sharp}_{\tilde A}) = dim(U_{A})$, and $U_{A}$ is the Galois group. Now $K^{\sharp}_{G}$ is a union of Picard-Vessiot extensions of $K^{\sharp}_{\tilde A} = K^{\sharp}_{U_{A}}$, each with Galois group $\Ga_{m}$ (by 3.3) so the Galois theory tells us that $z$ is independent from $K^{\sharp}_{G}$ over $K^{\sharp}_{\tilde A}$.  Hence the differential Galois group of $K^{\sharp}_{G}(z)$ over $K^{\sharp}_{G}$ is $U_{A}^{\partial}$. 
But then the Galois group of  $K^{\sharp}_{G}(y)$ over $K^{\sharp}_{G}$ will be the group of $\partial$-points of a $\partial$-subgroup of $H$ which projects onto $U_{A}$.  The only possibility is $H$ itself, because otherwise $H$ splits as $\Ga_{m}\times U_{A}$ as a $\partial$-group, which contradicts (v) of Section 2 of \cite{Bertrand-Fourier}. This completes the proof.  
\end{proof}

Essentially the same argument applies if we replace $H$ by the preimage under $\pi$ of some nontrivial $\partial$-subgroup of $U_{A}$.  So this shows that the scenario described right before Lemma 3.5, reduces to the case where $a\in LT$ where $T$ is the toric part $\Ga_{m}$ (of both $G$ and $H$), and we may assume $y\in T(K^{diff})$. We would like to show (in 
analogy with 3.5) that if $y\notin G(K) + G^{\partial}(K^{diff})$ then $tr.deg(K_{G}^{\sharp}(y)/K_{G}^{\sharp}) = 1$.  Of course already $K(y)$ is a Picard-Vessiot extension of $K$ with Galois group $T(\C)$, and we have to prove that $y$ is independent from $K^{\sharp}_{G}$ over $K$.  One deduces from the Galois theory that $y$ is independent from 
$K^{\sharp}_{U_{A}}$ over $K$. It remains to show that for any $z_{1},..z_{n}\in G^{\partial}(K^{diff})$, $y$ is independent from $z_{1},..,z_{n}$ over $K_{U_{A}}^{\sharp}$. If not, the discussion in Case (I) of the proof of Theorem 2.7, gives that $y = z+g$ for some $z\in G^{\partial}(K^{diff})$ and $g\in G(K_{U_{A}}^{\sharp})$, which does not suffice to yield a contradiction.  It would be enough in this section to prove a ``domination" statement, in the sense of model theory namely that $K_{U_{A}}^{\sharp}$ dominates $K_{G}^{\sharp}$ over $K$. Recall that this means that {\em anything}  independent from $K_{U_{A}}^{\sharp}$ over $K$ is independent from $K_{G}^{\sharp}$ over $K$.  We did not succeed in proving this yet, although it should be the case.

\vspace{2mm}

Similar and other issues arise when  we want to drop the restriction on the toric part.  For example in Case (ii) in the proof of Theorem 2.6, we can no longer deduce the splitting of $G$ as $T\times \tilde A$. And in the proof of Theorem 2.7,  both the analogues of Case (I) $H = T$ and Case (II) $H$ projects on to $\tilde A$,  present technical difficulties. 

\section{Lindemann-Weierstrass}
 
We here prove Corollaries  2.8, 2.9, and 2.10. 
\subsection{General results}

\subsubsection*{Proof of Corollary 2.8}
We first prove (i). Write $G$ for $\tilde B$.  Let $\tilde x \in LG(K)$ be a lift of $x$ and $\tilde y\in G(\cal U)$ a solution of $\partial\ell n_{G}(-) = \tilde x$.  We refer to  Section 1.2  
and Lemma 4.2 of \cite{LW} for a discussion of 
the equivalence of the hypotheses 
`` $x\notin LH(K) + LB_{0}(\C)$ for any proper semiabelian subvariety $H$ of $B$ "  and 
`` (*) $\tilde x \notin LH(K) + (LG)^{\partial}(K)$ for any proper algebraic subgroup $H$ of $G$ over $K$ ".
 
 \vspace{2mm}
Let $a = \partial_{LG}(\tilde x)$. So $\tilde y$ is a solution of the logarithmic differential equation (over $K$) $\partial\ell n_{G}(-) = a$. We want to show  that $tr.deg(K^{\sharp}_{G}(\tilde y)/K^{\sharp}_{G}) = \dim(G)$. If not, we may assume that $\tilde y\in G(K^{diff})$, and so by Theorem 2.6, $\tilde y\in H+G(K)$ for some proper connected algebraic  $\partial$-subgroup $H$ of $G$ defined over $K$. 
  Extend $H$ to a maximal proper connected $\partial$-subgroup $H_{1}$ of $G$, defined over $K$. Then $G/H_{1}$ is either (i) $\Ga_{m}$ or (ii) a simple abelian variety $A_{0}$ over $\C$, or (iii)  the quotient of $\tilde A_{1}$ by a maximal unipotent $\partial$-subgroup, where $A_{1}$ is a simple abelian variety over $K$ with $\C$-trace $0$. 
Let $x'$, $y'$ be the images of $\tilde x$, $\tilde y$ under the map $G\to G/H_{1}$  and induced  $LG \to L(G/H_{1})$. So both $x',y'$ are $K$-rational. Moreover the hypothesis  (*)  is preserved in $G/H_{1}$  (by our assumptions on $G$ and Lemma 4.2(ii) of \cite{LW}). 
As $\partial\ell n_{G/H_{1}}(y') = \partial_{L(G/H_{1})}(x')$, we have a contradiction in each of the cases (i), (ii), (iii) listed above, by virtue of the truth of Ax-Lindemann in the constant case, as well as  Manin-Chai (Proposition 4.4 in \cite{LW}). 

\vspace{2mm}
\noindent
(ii) Immediate as in  \cite{LW}:  Choosing $\tilde y = exp_{G}(\tilde x)$, then $exp_{B}(y)$ is the projection of $\tilde y$ on $B$.

\subsubsection*{Proof of Corollary 2.9}

This is like the proof of Corollary 2.8.  So $x\in LA(K)$. Let $\tilde x\in L\tilde A(K)$ lift $x$ and let $\tilde y\in \tilde A(K^{diff})$ be such that $\partial\ell n_{\tilde A}(\tilde y) = \partial_{L\tilde A}(\tilde x) = a$, say.   Let $B$ be a minimal abelian subvariety of $A$ such that $x\in LB(K) + L
A_{0}(\C)$, and we want to prove that 
$tr.deg(K^{\sharp}_{\tilde A}(\tilde y)/K^{\sharp}_{\tilde A}) = dim(\tilde B)$.
\newline
{\em Claim.} We may assume that $x\in LB(K)$, $\tilde x\in L\tilde B(K)$ and $\tilde y\in \tilde B(K^{diff})$. 
\newline
{\em Proof of claim.} Let $x= x_{1} + c$ for $x_{1}\in LB$ and $c\in LA_{0}(\C)$.  Let ${\tilde x_{1}}\in L\tilde B(K)$ be a lift of $x_{1}$ 
and $\tilde c\in L{\tilde A_{0}}(\C)$ be a lift of $c$. Finally let $\tilde y_{1}\in \tilde B(K^{diff})$ be such that $\partial\ell n_{\tilde A}(\tilde y_{1}) = \partial_{L\tilde A}(\tilde x_{1}) = 
a_{1}$, say. 
As ${\tilde x_{1}} + \tilde c$ projects on to  $x$, it differs from $\tilde x$ by an element $z\in LW(K)$.  Now $\partial_{L\tilde A}(z) = \partial\ell n_{\tilde A}(z)$. 
So $a = \partial_{L\tilde A}(\tilde x) = \partial_{L\tilde A}(\tilde x_{1} + \tilde c + z)  = \partial_{L\tilde A}(\tilde x_{1}) + \partial\ell n_{\tilde A}(z) = a_{1}  + \partial\ell n_{\tilde A}(z)$. Hence $\partial\ell n(\tilde y_{1} + z) = a$, and so $\tilde y_{1} + z$ differs from $\tilde y$, by an element of ${\tilde A}^{\partial}$. Hence $tr.deg(K^{\sharp}_{\tilde A}(\tilde y_{1})/K^{\sharp}_{\tilde A}) = tr.deg(K^{\sharp}_{\tilde A}(\tilde y_{1})/K^{\sharp}_{\tilde A})$. Moreover the same hypothesis remains true of $x_{1}$ (namely $B$ is minimal such that $x_{1}\in LB + LA_{0}(\C)$). So we can replace $x,\tilde x,\tilde y$ by $x_{1},\tilde x_{1}, \tilde y_{1}$. 

\vspace{2mm}
 
As recalled in the proof of Corollary 2.8   
(see Corollary H.5 of \cite{LW}), the condition that $x\notin B_{1}(K) + LA_{0}(\C)$ for any proper abelian subvariety $B_{1}$ of $B$ is equivalent to
(*) $\tilde x\notin LH(K) + (L\tilde A)^\partial(K)$ 
for any proper algebraic subgroup $H$ of $\tilde B$ defined over $K$.
Now  we can use the Galois-theoretic result  Proposition 2.5, namely the truth of Corollary 2.3 for $\tilde A$, as above.  That is, if by way of contradiction $tr.deg(K^{\sharp}_{\tilde A}(\tilde y)/K^{\sharp}_{\tilde A}) < dim(\tilde B)$, then $\tilde y \in H + \tilde A(K) + (\tilde A)^{\partial}(K^{diff}$ for some proper  
connected algebraic $\partial$-subgroup of $\tilde B$, defined over $K$,  and moreover $H^\partial$ is the differential Galois group of $K^{\sharp}_{\tilde A}(\tilde y)/K^{\sharp}_{\tilde A}$. As at the end of the proof of Corollary 2.8 above we get a contradiction by choosing $H_{1}$ to be a maximal proper connected algebraic $\partial$-subgroup of $\tilde A$, containing $H$ and defined over $K$. This concludes the proof   of 2.9. 

\subsection{Semiabelian surfaces}
We first recall the counterexample from Section 5.3 of \cite{LW}.  This example shows that in  Corollary 2.8, we cannot drop the assumption that the semiconstant part is constant.  We go through it again briefly. Let $B$ over $K$ be a nonconstant extension of a constant elliptic curve $E = E_0$  by $\Ga_{m}$, and let $ G = \tilde B$. Let $ \tilde x \in LG(K)$ map onto a point $\check x$ in
$ L\tilde E(\C)$ which itself maps onto a nonzero point $\bar x$  of $LE(\C)$. As pointed out in \cite{LW}  $(LG)^{\partial}(K) = (L\Ga_{m})(\C)$, whereby $\tilde x$ satisfies the hypothesis (*) from 2.8:  $\tilde x\notin LH(K) + (LG)^{\partial}(K)$  for any proper algebraic subgroup $H$ of $G$. 
Let $a= \partial_{LG}(\tilde x)\in LG(K)$, and $\tilde y\in G(K^{diff})$ such that $\partial\ell n _{G}( \tilde y) = a$.  Then as the image of $a$ in $L\tilde E$ is $0$, $ \tilde y$ projects onto a 
point of $\tilde E(\C)$, and hence $\tilde y$ is in a coset of $\Ga_{m}$ defined over $K$ whereby $tr.deg(K( \tilde y)/K) \leq 1$, so a fortiori the same with $K^{\sharp}_{G}$ in place of $K$.
A consequence of Corollary 2.10, in fact the main part of its proof,  is that with  the above choice of $\tilde x$, we have $tr.deg(K^{\sharp}_{G}(\tilde y)/K^{\sharp}_{G}) = 1$ (as announced in \cite{BMPZ}, Footnote 5).

\subsubsection*{Proof of Corollary 2.10}

Let us fix notations: $B$ is a semiabelian variety over $K$ with toric part $\Ga_{m}$ and abelian quotient  a non-necessarily constant elliptic curve $E/K$, with constant part $E_0$; $G$ denotes the universal vectorial extension $\tilde B$ of $B$ and $\tilde E$ the universal vectorial extension of $E$. For $x\in LB(K)$, $\tilde x$ denotes a lift of $x$ to a point of $LG(K)$,  $\bar x$ denotes the projection of $x$ to $LE(K)$, and $\check x$  denotes the projection of $\tilde x$ to $L\tilde E(K)$. 

Recall the hypothesis (*) in 2.10: $x\notin LH + LB_{0}(\C)$ for any proper algebraic subgroup $H$ of $B$.  As pointed out after the statement of Corollary 2.10, under this hypothesis, the condition $\bar x \in LE_{0}(\C)$ can occur only if $B$ is semiconstant and not constant.  Indeed,  if $B$ were not semiconstant then $E_{0} = 0$ so $x\in L\Ga_{m}$ contradicting the hypothesis on $x$. And if $B$ is constant then $B = B_{0}$  and $\bar x$ has a lift in $LB_{0}(\C)$ 
 whereby $x\in L\Ga_{m} + LB_{0}(\C)$, contradicting the hypothesis.

Now if the semiconstant part of $B$ is constant, then we can simply quote Corollary 2.8, bearing in mind the paragraph above which rules out the possibility that $\bar x \in LE_{0}(\C)$. So we will assume that $B_{sc} \neq  B_{0}$,  namely $E= E_{0}$ and $B_{0} = \Ga_{m}$. 

\vspace{2mm}
\noindent
{\em Case (I).}  $\bar x\in LE(\C)$ ($= LE_{0}(\C)$ as $E = E_{0}$). 
\newline
This is where the bulk of the work goes.  We first check that we are essentially in the situation of the ``counterexample" mentioned above. The argument is a bit like in the proof of the claim in Corollary 2.9. Note that $\bar x \neq 0$ by hypothesis (*). 
Let $\check x'$ be a lift of $\bar x$ to a point in $L\tilde E(\C)$  (noting that $\tilde E$ is also over $\C$).  Then $\check x' = \check x - \beta$ for some $\beta\in L\Ga_{a}(K)$. 
Let $\tilde x' = \tilde x - \beta$. Let $a' = \partial_{LG}(\tilde x')$. Then (as $\partial_{LG}(\beta) = \partial\ell n_{G}(\beta)$, under the usual identifications) $a' = a + \partial\ell n_{G}(\beta)$, and if $\tilde y'\in G$ is such that $\partial\ell n_{G}(\tilde y') = a'$ then $\partial\ell n_{G}(\tilde y' -\beta) = a$.  As $\beta\in G(K)$, $tr.deg(K^{\sharp}_{G}(\tilde y')/K^{\sharp}_{G}) = tr.deg(K^{\sharp}_{G}(\tilde y)/K^{\sharp}_{G})$.
\newline
The end result is that we can assume that $\tilde x\in LG(K)$ maps onto $\check x'\in L\tilde E(\C)$ which in turn maps on to our nonzero $\bar x \in LE(\C)$, precisely the situation in the example above from Section 5.1 of \cite{LW}. So to deal with Case (I), we have to prove:

\vspace{2mm}
\noindent
{\em Claim 1.} $tr.deg(K^{\sharp}_{G}(\tilde y)/K^{\sharp}_{G}) = 1$. 
\newline
{\em Proof of claim 1.} 
Remember that $a$ denotes $\partial_{LG}(\tilde x)$. 
Now by Theorem 2.7, it suffices to prove that $a\notin \partial\ell n_{G}(G(K))$. 

We assume for a contradiction that there is $\tilde s\in G(K)$ such that
$$ a = \partial _{LG}(\tilde x) = \partial\ell n_{G}(\tilde s).  \qquad\quad (\dagger)$$
This is the semi-abelian analogue of a Manin kernel statement, which can probably be studied directly, but we will deduce the contradiction from \cite{BMPZ}. Let $\tilde x_{1} = log_{G}(\tilde s)$ be a solution given by complex analysis to the linear inhomogeneous equation  $\partial_{LG}(-) = \partial\ell n_{G}(\tilde s)$.  Namely, with notations as in the appendix to \cite{LW} (generalizing those given after Corollary 2.8 above),  a local analytic section of $L{\bf G}^{an}/S^{an}$ such that $exp_{\bf G}(\tilde x_{1}) = \tilde s$.  Let $\xi\in (LG)^{\partial}$ be $\tilde x- \tilde x_{1}$. Then $\xi$ lives in a differential field (of meromorphic functions on some disc in $S$) which extends $K$ and has the same constants as $K$, namely $\C$. As $\xi$ is the solution of a linear homogeneous differential equation over $K$, $\xi$ lives in $(LG)^{\partial}(K^{diff})$. Hence as $\tilde x\in LG(K)$ this implies that $\tilde x_{1}\in LG(K^{\sharp}_{LG})$  where $K^{\sharp}_{LG}$ is the differential field generated over $K$ by $(LG)^{\partial}(K^{diff})$. 

Now from Section 5.1 of \cite{BMPZ}, $K^{\sharp}_{LG}$ coincides with the ``field of periods" $F_{q}$ attached to the point 
$q \in {\hat E}(K)$ which parametrizes the extension $B$ of $E$ by $\Ga_{m}$.  Hence from ($\dagger$) we conclude
that $ F_{q}(log_{G}(\tilde s)) = F_{q}.$

 Let $s\in B(K)$ be the projection of $\tilde s$, and let $p\in E(K)$ be the projection of $s$. By \cite{BMPZ}, discussion in Section 5.1,  we have that 
$F_{pq}(log_{B}(s)) = F_{q}(log_{G}(\tilde s))$.  Therefore, $ F_{q} = F_{pq} = F_{pq}(log_{B}(s))$. 

Now as $\tilde x\in LG(K)$ maps onto the constant point $\check x \in L\tilde E(\C)$, so also $\tilde s$ maps onto a constant point  $\check p \in \tilde E(\C)$ and hence $p\in E(\C)$. So we are in Case (SC2) of the 
proof of the 
Main Lemma of \cite{BMPZ}, \S 6,  namely $p$ constant while $q$ nonconstant.  The conclusion of (SC2) is that $log_{B}(s)$ is transcendental over $F_{pq}$ if $p$ is nontorsion. So the previous equality forces $p\in E(\C)$ to be torsion. 

Let $\tilde s_{tor}\in G(K)$ be a torsion point lifting $p$, hence  $\tilde s - \tilde s_{tor}$ is a $K$-point of the kernel of the surjection $G\to E$. Hence $\tilde s = \tilde s_{tor} + \delta + \beta$ where
 $\beta \in \Ga_{a}(K)$ and $\delta\in \Ga_{m}(K)$.  Taking logs,  putting again $\xi = \tilde x-\tilde x_{1}$ ,  and  using that  $log_{G}(-)$ restricted to $\Ga_{a}(K)$ is the identity, we see 
that $\tilde x = \xi + log_{G}(\tilde s_{tor}) +  log_{G}(\delta)  + \beta = \xi' + log_{\Ga_{m}}(\delta) + \beta$  where $\xi'\in (LG)^{\partial}$.
It follows that $\ell = log_{\Ga_{m}}(\delta) \in K^{\sharp}_{G} = F_{q}$. But by Lemma 1 of \cite{LW} (proof of Main Lemma  in isotrivial case, but reversing roles of $p$ and $q$), such $\ell$ is transcendental over $F_{q}$ unless $\delta$ is constant.

Hence $\delta \in\Ga_{m}(\C)$, whereby $log_{\Ga_{m}}(\delta) \in L\Ga_{m}(\C)$ so is in $(LG)^{\partial}(K^{diff})$, and we conclude that 
$\tilde x-\beta \in (LG)^{\partial}(K^{diff})$. As also $\tilde x-\beta \in LG(K)$, from Claim III in Section 5.3 of \cite{LW} 
(alternatively, using the fact that $K^\sharp_{LG} = F_q$ has transcendence degree $2$ over $K$),
 we conclude that 
$\tilde x-\beta \in L\Ga_{m}(\C)$ whereby $\tilde x\in L\Ga_{a}(K) + L\Ga_{m}(\C)$, contradicting that $x$ projects onto a nonzero element $LE$. This contradiction completes the proof of  Claim 1  and hence of Case (I) of Corollary 2.10.  

\vspace{2mm}
\noindent
{\em Case (II).} $\bar x \in LE(K)\setminus LE(\C)$ is a nonconstant point of $LE(K) = LE_{0}(K)$. 
\newline
Let $\tilde y\in G(K^{diff})$ be such that  $\partial\ell n_{G}(\tilde y) = a = \partial_{LG}(\tilde x)$.  Let $\check y$ be the projection of $\tilde y$ to $\tilde E$. Hence 
$\partial\ell n_{\tilde E}(\check y) = \partial_{L\tilde A}(\check x)$ (where remember $\check x$ is the projection of $\tilde x$ to $L\tilde E$).  Noting that $\check x$ lifts $\bar x\in LE(K)$, and using our case hypothesis, we can apply Corollary 2.9 to $E$ to conclude that  $tr.deg(K(\check y)/K) = 2$ with Galois group  ${\tilde E}^\partial(K^{diff}) = \tilde E(\C)$. (In fact as $E$ is constant this is already part of the Ax-Kolchin framework and appears in \cite{Bertrand-LMS}.) 

\vspace{2mm}
\noindent
{\em Claim 2.} $tr.deg(K^{\sharp}_{G}(\check y)/K^{\sharp}_{G}) = 2$. 
\newline
{\em Proof of Claim 2.}  Fact 3.4 applies to the current situation, showing that $K^{\sharp}_{G}$ is a directed union of Picard-Vessiot extensions of $K$ each with Galois group some 
product of $\Ga_{m}^{n}(\C)$'s.  As  there are no  
proper algebraic subgroups of $\tilde E(\C)\times \Ga_{m}^{n}(\C)$ projecting onto each factor, it follows from the Galois theory, that $\check 
y$ is independent from $K^{\sharp}_{G}$ over $K$, yielding Claim 2. 

\vspace{2mm}
Now $K^{\sharp}_{G}(\tilde y)/K^{\sharp}_{G}$ is a differential Galois extension with Galois group of the form $H^{\partial}(K^{diff})$ where $H$ is  connected algebraic $\partial$-subgroup of $G$. So $H^{\partial}$ projects  onto the (differential) Galois group of $K^{\sharp}_{G}(\check y)$ over $K^{\sharp}_{G}$, which by Claim 2 is ${\tilde E}^{\partial}(K^{diff})$. In particular $H$ projects onto $\tilde E$. If $H$ is a proper subgroup of $G$, then projecting $H$, $\tilde E$ to $B$, $E$, respectively, shows that $B$ splits (up to isogeny), so $B = B_{0}$ is constant, contradicting the current assumptions.
Hence the (differential) Galois group of $K^{\sharp}_{G}(\tilde y)$ over $K^{\sharp}_{G}$ is $G^{\partial}(K^{diff})$, whereby 
$tr.deg(K^{\sharp}_{G}(\tilde y/K^{\sharp}_{G})$ is $3$. This concludes the proof of Corollary 2.10.

\subsection{An Ax-Schanuel conjecture}

As a conclusion to the first two themes of the paper, we may say that both at the Galois theoretic level and for Lindemann-Weierstrass, we have obtained rather definitive results  for families of abelian varieties, and working over a suitable base $K^{\sharp}$. There  remain open questions for families of semiabelian varieties, such as Conjecture 2.3, as well as dropping the restriction on the toric part in 2.6, 2.7, 2.8, and 2.10.  It also remains to formulate a qualitative description of $tr.deg(K^{\sharp}(exp_B(x)/K^{\sharp})$ where $B$ is a semiabelian variety over $K$ of dimension $> 2$, and $x\in LB(K)$, under the nondegeneracy hypothesis that $x\in LH + LB_{0}(\C)$ for any proper semiabelian subvariety $H$ of $B$. 

\vspace{2mm}

Before turning to our third theme, it seems fitting to  propose a more general {\em Ax-Schanuel}  conjecture for families of abelian varieties, as follows.
\begin{Conjecture} Let $A$ be an abelian variety over $K = \C(S)$ for a curve $S/\C$, and let $F$ be the field of meromorphic functions on some disc in $S$. Let $K^{\sharp}$ now denote $K^{\sharp}_{L\tilde A}$ (which contains $K^{\sharp}_{\tilde A}$).  Let $\tilde x$, $\tilde y$ be $F$-rational points of $L\tilde A, \tilde A$ respectively such that $exp_{\tilde A}(\tilde x) = \tilde y$, and let $y$
 be the projection of $\tilde y$ on $A$. Assume that $y\notin H + A_{0}(\C)$ for any proper 
 algebraic subgroup $H$ of $A$. Then $tr.deg (K^{\sharp}(\tilde x, \tilde y)/K^{\sharp}) \geq dim(\tilde A)$. 
\end{Conjecture} 

We point out that the assumption concerns $y$, and not the projection $x$ of $\tilde x$ to $LA$. Indeed, the conclusion would in general not hold true under the weaker hypothesis that   $x\notin LH + LA_{0}(\C)$ for any proper abelian subvariety $H$ of $A$. 
As a counterexample, take for $A$ a simple non constant abelian variety over $K$, and for $\tilde x$  a non-zero period of $L\tilde A$. Then, $x \neq 0$ satisfies the hypothesis above and $\tilde x$ is defined over $K^\sharp = K^{\sharp}_{L\tilde A}$, but $\tilde y = exp_{\tilde A}(\tilde x) = 0$, so  $tr.deg (K^{\sharp}(\tilde x, \tilde y)/K^{\sharp}) = 0$. 

\vspace{2mm}

Finally, here is a concrete corollary of the conjecture. Let ${\bf E} : y^2 = x(x-1)(x-t)$ be the universal Legendre elliptic curve over $S = \C \setminus \{0, 1\}$, and let $\omega_1(t), \omega_2(t)$ be a basis of the group of periods of $\bf E$ over some disk, so $K^{\sharp} = K^{\sharp}_{L\tilde E}$ is the field generated over $K = \C(t)$ by $\omega_1, \omega_2$ and their first derivatives. Let $\wp = \wp_t(z), \zeta = \zeta_t(z)$ be the standard Weierstrass functions attached to $\{\omega_1(t), \omega_2(t)\}$. For $g \geq 1$, consider $2g$ algebraic functions $\alpha_1^{(i)}(t), \alpha_2^{(i)}(t) \in K^{alg}, i = 1, ..., g$, and assume that the vectors 
$  \left( { \begin{array}{ccccccc} 1   \\  0 \end{array}} \right),  \left( { \begin{array}{ccccccc} 0   \\  1 \end{array}} \right),  \left( { \begin{array}{ccccccc} \alpha_1^{(1)}  \\ \alpha_2^{(1)}  \end{array}} \right), ..., \left( { \begin{array}{ccccccc} \alpha_1^{(g)}  \\ \alpha_2^{(g)}  \end{array}} \right) $ are linearly independent over $\Z$. Then, the $2g$ functions 
$$\wp(\alpha_1^{(i)}\omega_1 + \alpha_2^{(i)}\omega_2) ~,~ \zeta(\alpha_1^{(i)}\omega_1 + \alpha_2^{(i)}\omega_2)~, i = 1, ..., g,$$
 of the variable $t$ are algebraically independent over $K^\sharp$. In the language of \cite{BMPZ}, \S 3.3, this says in particular  that  a $g$-tuple of  $\Z$-linearly independent  local analytic sections of  ${\bf E}/S$ with algebraic {\it Betti} coordinates forms a generic point of ${\bf E}^g/S$. Such a statement is not covered by our Lindemann-Weierstrass results, which concern  analytic sections with algebraic logarithms.

\section{Manin maps}

\subsection{Injectivity}
We here prove Theorem  2.13. and Proposition 2.15. Both statements will follow fairly quickly  from Fact 5.1 below, which is Theorem 4.3 of \cite{Bertrand-descent} and relies on the strongest version of ``Manin-Chai", namely formula $(2^{*})$
 from Section 4.1 of \cite{Bertrand-descent}.
We should mention that a more direct proof of Proposition 2.15 can be extracted from the proof of Proposition J.2 (Manin-Coleman) in \cite{LW}. But we will stick with the current proof below, as it provides a good introduction to the counterexample in Section 5.2.

 \vspace{2mm}
 
 We set up notations :   $K$ is $\C(t)^{alg}$ as usual,  $A$ is an an abelian variety over $K$ and $A_{0}$ is the $\C$-trace of $A$. For $y\in \tilde A(K)$, we let $\overline y$  be its image in $A(K)$. Let $b = \partial\ell n_{\tilde A}(y)$.
  We consider the differential system in unknown $x$:
 $$\nabla_{L\tilde A}(x) = b  ~,$$
 where we write $\nabla_{L\tilde A}$ for $\partial_{L\tilde A}$.  Let $K^{\sharp}_{L\tilde A}$ be the differential field generated, over $K$, by 
 $(L\tilde A)^{\partial}(K^{diff})$. So for $x$ a solution   in $L\tilde A(K^{diff})$, the differential Galois group of $K^{\sharp}_{L\tilde A}(x)$ over $K^{\sharp}_{L\tilde A}$ pertains to Picard-Vessiot theory, and is well-defined as a $\C$-subpace of  the $\C$-vector space $(L\tilde A)^{\partial}(K^{diff})$.

\begin{Fact} [$A = $ any abelian variety over $K = \C(t)^{alg}$] Let $y\in \tilde A(K)$. Let $B$ be the smallest abelian subvariety of $A$ such that a multiple of $\overline y$  by a nonzero integer is in  $B + A_{0}(\C)$.  Let $x$ be a solution of $\nabla_{L\tilde A}(-) = b$ in $L\tilde A(K^{diff})$. 
Then the differential Galois group of $K^{\sharp}_{L\tilde A}(x)$ over $K^{\sharp}_{L\tilde A}$ is $(L\tilde B)^{\partial}(K^{diff})$. In particular  $tr.deg(K^{\sharp}_{L\tilde A}(x)/K^{\sharp}_{L\tilde A}) = dim\tilde B  = 2dimB$. 
\end{Fact}

\subsubsection*{Proof of Theorem 2.13} 
Here, $A$ has $\C$-trace $0$. 
By assumption we have $y_{1},..,y_{n}\in A(K)$ and $a_{1},..,a_{n}\in \C$, not all $0$ such that 
$a_{1}\mu_{A}(y_{1}) + ... + a_{n}\mu_{A}(y_{n}) = 0$ in $L\tilde A(K)/\partial\ell n_{\tilde A}(W_{A})$. Lifting $y_{i}$ to $\tilde y_{i} \in \tilde 
A(K)$, we derive that
$$ a_{1}\partial\ell n_{\tilde A}(\tilde y_{1}) + .... + a_{n}\partial\ell n_{\tilde A}(\tilde y_{n}) = \partial\ell n_{\tilde A}(z)$$
for some $z\in W_{A}$.   Via our identification of $W_{A}$ with $LW_{A}$ we write the right hand side as $\nabla_{L\tilde A}z$ with $z\in 
LW_{A}\subset L\tilde A$.  Let $\tilde x_{i}\in L\tilde A$ be such that $\nabla_{L\tilde A}(\tilde x_{i}) = \partial\ell n_{\tilde A}(\tilde y_{i})$. 
Hence $a_{1}\tilde x_{1} + ... + a_{n}\tilde x_{n} - z \in (L\tilde A)^{\partial}$, and there exists $d\in (L\tilde A)^{\partial}$ such that
$$a_{1}\tilde x_{1} + ... + a_{n}\tilde x_{n} - d = z  \in LW_A .$$
Suppose for a contradiction that $y_{1},..,y_{n}$ are linearly independent with respect to $End(A)$. Then  no  multiple of $y = (y_{1},..,y_{n})$ by a nonzero integer lies in any  proper abelian subvariety $B$ of the traceless abelian variety $A^{n} = A\times .. \times A$.  By Fact 5.1, 
$tr.deg(K^\sharp(\tilde x_{1},..,\tilde x_{n})/K^\sharp) = dim(\tilde A^n)$, where we have set  $K^\sharp := K^{\sharp}_{L\tilde{A^{n}}} = K^{\sharp}_{L\tilde A}$. 
So $\tilde x_{1},..,\tilde x_{n}$ are generic independent, over $K^\sharp$, points of $L\tilde A$. Hence, as $a_{1},..,a_{n}$ are in $\C$ so in $K^\sharp$, $a_{1}\tilde x_{1} + ... + a_{n}\tilde x_{n}$ is a generic point of $L\tilde A$ over $K^\sharp$. And as $d$ is a $K^\sharp$-rational point of $(L\tilde A)^{\partial}$,  $a_{1}\tilde x_{1} + .. + a_{n}\tilde x_{n} - d = z$ too is a generic point of $L \tilde A$ over $K^\sharp$, so cannot lie in its strict subspace $LW_A$. This contradiction concludes the  proof of 2.13. 

\subsubsection*{Proof of Proposition 2.15.} 
 We use the same notation as at the end of Section 2.4,  and recall that $A$ is traceless. Furthermore, the functoriality of $M_{K,A}$ in $A$ allows us to assume that $A$ is a simple abelian variety. 

\vspace{2mm}
\noindent
{\it Step (I).} We show as in the proof of 2.13 above that if $M_{K,A}(y_{1}),..,M_{K,A}(y_{n})$ are $\C$-linearly dependent, then $y_{1},..,y_{n}$ are $End(A)$-linearly dependent. 
Indeed, assume that $a_{i}\in \C$ are not all $0$ and that  $a_{1}M_{K,A}(y_{1}) +  .. + a_{n}M_{K,A}(y_{n}) = 0$ in  the target space  $L\tilde A(K)/\nabla(L\tilde A(K))$. 
Lifting $y_{i}$ to $\tilde y_{i}\in \tilde A(K)$, we derive that
$$a_{1}\partial\ell n_{\tilde A}(\tilde y_{1}) + .... + a_{n}\partial\ell n_{\tilde A}(y_{n}) \in \nabla(L\tilde A (K))$$
Letting $\tilde x_{i}\in L\tilde A(K^{diff})$ be such that $\nabla\tilde x_{i} = \partial\ell n_{\tilde A}(\tilde y_{i})$, we obtain a $K$-rational point $z\in L\tilde A(K)$ such that
$$    a_{1}\tilde x_{1} + ... + a_{n}\tilde x_{n} - z := d \in (L\tilde A)^{\partial}(K^{diff}).$$
Taking $K^\sharp := K^\sharp_{L\tilde A}$ as in the proof of 2.13, we get $tr.deg(K^\sharp(\tilde x_{1},..,\tilde x_{n})/K^\sharp) < dim({\tilde A}^{n})$. Hence by Fact 5.1,
 some integral multiple of  $(y_{1},..,y_{n})$ lies in a proper abelian
subvariety of $A^{n}$, whereby $y_{1},..,y_{n}$ are $End(A)$-linearly dependent. 

\vspace{2mm}
\noindent
{\it Step (II).}  Assuming that $y_{1},..,y_{n}$ are $End(A)$-linearly dependent, given by Step (I), as well as the relation on the point $d$ above with not all $a_{i}= 0$, we will show that the points $y_{i}$ are $\Z$-linearly dependent.  Equivalently we will show that if a similar relation holds with the $a_{i}$ linearly independent over $\Z$, then $y =(y_{1},..,y_{n})$ is a torsion point of $A^n$.
Let $\tilde x = (\tilde x_{1},...,\tilde x_{n})$.  Let $B$ be the connected component of the 
Zariski closure of the
group $\Z\cdot y$ of multiples of $y$ in $A^{n}$.  By Fact 5.1,  the  differential  Galois group of $K^\sharp(\tilde x)$ over $K^\sharp := K^\sharp_{L\tilde A}$ is 
$(L\tilde B)^\partial$. More precisely, the set of $\sigma(\tilde x) - \tilde x$ as $\sigma$ varies in $Aut_\partial(K^\sharp(\tilde x)/K^\sharp)$ is precisely $(L\tilde B)^{\partial} \subseteq (L\tilde A^{n})^{\partial}$.  Since $z$ and $d$ are defined over $K^\sharp$,  the relation on $d$ implies that
$$ \forall (\tilde c_{1},..,\tilde c_{n})\in (L\tilde B)^{\partial}, a_{1}\tilde c_{1} + ... + a_{n}\tilde c_{n} = 0. $$
Let now
$${\cal B}   = \{\alpha = (\alpha_{1},..,\alpha_{n}) \in (End(A))^{n} = Hom(A, A^n):  \alpha(A) \subseteq B \subset A^{n}\}.$$ 

\vspace{2mm}
\noindent
{\em Claim.} Assume that $a_1, ..., a_n$ are linearly independent over $\Z$. Then, any $\alpha\in {\cal B}$ is identically $0$.

\vspace{2mm}
\noindent
It follows from the claim that $B = 0$ and hence some multiple of $y$ by a nonzero integer vanishes, namely $y$ is a torsion point of $A^{n}$. This completes the proof of Step (II), hence of Proposition 2.15, and we are now reduced to proving the claim. 

\vspace{2mm}
\noindent
{\it Proof of claim.} Since $A$ is simple, $End(A)$ is an order in a simple algebra $D$ over $\Q$. For $i=1,..,n$, denote by $\rho(\alpha_{i})$ the $\C$-linear map induced on $(L\tilde A)^{\partial}$ by the endomorphism $\alpha_{i}$ of $A$. So we view $(L\tilde A)^{\partial}$ as a complex representation, of degree $2dimA$,  of the $\Z$-algebra $End(A)$, or more generally, of $D$. Let $f^2$ be the dimension of $D$  over its centre $F$,  let $e$ be the degree of $F$ over $\Q$ and let $R$ be a reduced representation of $D$, viewed as a complex representation of degree $ef$. As the representation $\rho$ is defined over $\Q$
(since it preserves the Betti homology), 
$\rho$ is equivalent to the direct sum $R^{\oplus r}$ of  $r = 2dimA/ef$ copies of $R$ (cf. \cite{Sh-Ta}, \S 5.1). Furthermore,  $R : D \to Mat_f(F \otimes \C) \simeq (Mat_{f}(\C))^{e} \subset Mat_{ef}(\C)$ extends by $\C$-linearity  to an injection $R \otimes 1 : D\otimes \C \simeq (Mat_{f}(\C))^{e} \subset Mat_{ef}(\C)$.

\vspace{2mm}
Recall now that for any $(\tilde c_1, ..., \tilde c_n)$ in $(L\tilde B)^\partial$, $a_1 \tilde c_1 + ... + a_n \tilde c_n = 0$. Applied to the image under $\alpha = (\alpha_1, ..., \alpha_n) \in {\cal B}$ of the generic element of $(L\tilde A)^{\partial}$, this relation implies that
$$ a_{1}\rho(\alpha_{1}) + ... + a_{n}\rho(\alpha_{n}) = 0 \in End_{\C}((L\tilde A)^{\partial})$$
So $a_{1}R(\alpha_{1}) + .. + a_{n}R(\alpha_{n}) = 0$ in $(Mat_{f}(\C))^{e}$.  From the injectivity of $R \otimes 1$ on $D\otimes \C$ and the $\Z$-linear independence  of the $a_{i}$, we derive that each $\alpha_{i}\in D$ vanishes, hence $\alpha = 0$, proving the claim.

\subsection{A counterexample}
We conclude with the promised counterexample to the injectivity of $\mu_A\otimes 1$, namely 
Proposition 2.14.

\subsubsection*{Construction of $A$}
We will use Yves Andr\'e's example of a simple traceless abelian variety $A$ over $\C(t)^{alg}$ with $0 \neq U_A \subsetneq W_A$, cf. \cite{LW}, just before Remark 3.10. Since $U_A \neq W_A$, this $A$ is not constant, but we will derive this property 
and the simplicity of $A$ from another argument, borrowed from  \cite{LB}, IX.6. 

\vspace{2mm}

We start with a CM field $F$ of degree $2k$ over $\Q$, over a totally real number field $F_0$ of degree $k \geq 2$, and denote by $\{\sigma_1, \bar \sigma_1, ..., \sigma_k, \bar \sigma_k\}$ the complex embeddings of $F$. We further fix the CM type $S := \{\sigma_1, \bar \sigma_1, 2\sigma_2, ..., 2 \sigma_k\}$. By \cite{LB}, IX.6, we can attach to $S$ and to any $\tau \in {\cal H}$ (the Poincar\'e half-plane, or equivalently, the open unit disk)  an abelian variety $A = A_\tau$ of dimension $g = 2k$ and an embedding of $F$ into $End(A)\otimes \Q$ such that the representation $r$ of $F$ on $W_A$ is given by the type $S$. The representation $\rho$ of $F$ on $L\tilde A$ is then  $r \oplus \bar r$, equivalent to twice the regular representation. (The notations used by \cite{LB} here read~:  $e_0 = k, d = 1, m = 2, r_1 = s_1 = 1, r_2 = .. = r_{e_0} = 2, s_2 = ... = s_{e_0} = 0$, so, the product of the  ${\cal H}_{r_i, s_i}$  of {\it loc. cit.} is just $\cal H$. Also, \cite{LB} considers the more standard ``analytic" representation of $F$ on the Lie algebra $LA = L\tilde A/W_A$, which is $\bar r$ in our notation.)

From the bottom of \cite{LB}, p. 271, one infers that the moduli space of such abelian varieties $A_\tau$ is an analytic curve  ${\cal H}/\Gamma$. But  Shimura has shown that  it can be compactified to an algebraic curve $\cal X$, cf \cite{LB}, p. 247.  So, we can view the universal abelian variety  $A_\tau = A$ of this moduli space as an abelian variety   over $\C({\cal X})$, hence as an abelian variety $A$ over $K= \C(t)^{alg}$. This will be our $A$ : it is by construction not constant - and it is a fourfold if  we  take, as we will in what follows,  $k = 2$.

Finally,  since $A$ is the general element over ${\cal H}/\Gamma$,  Theorem 9.1 of  \cite{LB} and the hypothesis $k \geq 2$ imply that $End(A) \otimes \Q$ is {\it equal} to $F$. Therefore, $A$ is a simple abelian variety, necessarily traceless since it is not constant.  We denote by $\cal O$ the order $End(A)$ of $F$.

\subsubsection*{Action of $F$ and of $\nabla$ on $L \tilde A$}
For simplicity, we will now restrict to the case $k = 2$, but the general case (requiring $2k$ points) would work in exactly the same way. So, $F$ is a totally imaginary quadratic extension of a real quadratic field $F_0$, and $L \tilde A$ is 8-dimensional. As said in \cite{LW}, and by definition of the CM-type $S$, the action $\rho$ of $F$ splits $L \tilde A$ into eigen-spaces for its irreducible representations $\sigma$'s, as follows : 

\smallskip

 -  $W_A = D_{\sigma_1} \oplus  D_{\bar \sigma_1} \oplus P_{\sigma_2}$, where the $D$'s are lines , and  $P_{\sigma_2}$ is a plane;
 
 \smallskip
 
 - $LA$ lifts to $L \tilde A$ into $D'_{\sigma_1} \oplus  D'_{\bar \sigma_1} \oplus P_{\bar \sigma_2}$, with same notations.
 
 \medskip
 
Since $\nabla := \nabla_{L\tilde A} = \partial_{L\tilde A}$ commutes with the action $\rho$ of $F$ and since $A$ is not constant, we infer that the maximal $\partial$-submodule of $W_A$ is 
$$U_A = P_{\sigma_2} ~,$$
while $W_A + \nabla(W_A) = \Pi_{\sigma_1}�\oplus  U_A \oplus \Pi_{\bar \sigma_1}$, with planes $\Pi_{\sigma_1} = D_{\sigma_1}�\oplus D'_{\sigma_1}, \Pi_{\bar \sigma_1} = D_{\bar \sigma_1}�\oplus D'_{\bar \sigma_1} $, each stable under $\nabla$ (just as is $P_{\bar \sigma_2}$, of course). In fact, for our proof, we  only need to know that $P_{\sigma_2} \subset U_A$. �

\medskip
Now, let $\tilde y \in \tilde A(K)$ be  a lift of a point $y \in A(K)$. Going into a complex analytic setting, we choose a logarithm $\tilde x \in L \tilde A(K^{diff})$ of $\tilde y$, locally analytic on a small disk in ${\cal X}(\C)$. Let further $\alpha \in {\cal O}$, which canonically lifts to $End(\tilde A)$. Then, $\rho(\alpha) \tilde x$ is a logarithm of $\alpha. \tilde y \in \tilde A(K)$, and thefore satisfies 
$$\nabla (\rho(\alpha) \tilde x) = \partial \ell n_{\tilde A}(\alpha . \tilde y).$$
In fact, this appeal to analysis is not  necessary : the formula just says that $\partial \ell n_{\tilde A}$ (and $\nabla$) commutes with the actions of $\cal O$. But once one $\tilde y$ and one $\tilde x$ are chosen, it will be crucial for the searched-for  relation $({\frak R})$ following Proposition 2.14  that we take these $\rho(\alpha) \tilde x$ as solutions to the equations on the $\cal O$-orbit of $\tilde y$.

\medskip
Concretely, if
$$ \tilde x = x_{\sigma_2} \oplus  x_{\sigma_1}�\oplus x_{\bar \sigma_1}�\oplus x_{\bar \sigma_2}$$
is the decomposition of $\tilde x $ in  $L \tilde A =  P_{\sigma_2 } \oplus \Pi_{\sigma_1} \oplus \Pi_{\bar \sigma_1}   \oplus P_{\bar \sigma_2}$, then  for any $\alpha \in {\cal O}$, we have
$$\rho(\alpha)(\tilde x) =  \sigma_2(\alpha) x_{\sigma_2} \oplus  \sigma_1(\alpha) x_{\sigma_1} \oplus \bar  \sigma_1(\alpha) x_{\bar \sigma_1} \oplus   \bar \sigma_2(\alpha) x_{\bar \sigma_2} .$$

\subsubsection*{Conclusion}

\medskip
\noindent
Let $y \in A(K)$ be a non torsion point of the simple abelian variety $A$, for which we choose at will a lift $\tilde y$ to $\tilde A(K)$ and  a logarithm $\tilde x \in L\tilde A(K^{diff})$. Let $\{\alpha_1, ..., \alpha_4\}$ be an integral basis of $F$ over $\Q$. We will consider the 4 points $y_i = \alpha_i. y$ of $A(K)$, $i = 1, ..., 4$. Since the action of ${\cal O}$ on $A$ is faithful, they are linearly independent over $\Z$. For each $ i = 1, .., 4$,  we consider the lift  $\tilde y_i = \alpha_i \tilde y$ of $y_i$ to $L \tilde A(K)$, and set as above $\tilde x_i = \rho(\alpha_i) \tilde x$, which satisfies $\nabla(\tilde x_i) = \partial \ell n_{\tilde A} \tilde y_i$.  

\smallskip

We claim that there exist  complex numbers $a_1, .., a_4$, not all zero,  such that 
$$u := a_1 \tilde x_1 + ... + a_4 \tilde x_4 = \big(a_1 \rho(\alpha_1) + ... + a_4 \rho(\alpha_4)\big)({\tilde x})$$
 lies in $U_A(K^{diff})$, i.e. such that in the decomposition above of $L\tilde A =   P_{\sigma_2 } \oplus \Pi_{\sigma_1} \oplus \Pi_{\bar \sigma_1}   \oplus P_{\bar \sigma_2}$, the components of $u = u_{\sigma_2} \oplus  u_{\sigma_1}�\oplus u_{\bar \sigma_1}�\oplus u_{\bar \sigma_2}$ on the last three planes vanish.

\smallskip
The whole point is that the complex representation $\hat \sigma^{\oplus 2}$ of $F$ which $\rho$ induces on $ \Pi_{\sigma_1} \oplus \Pi_{\bar \sigma_1}   \oplus P_{\bar \sigma_2}$  is twice the representation $\hat \sigma  := \sigma_1 \oplus \bar \sigma_1 \oplus \bar \sigma_2$ of $F$ on $\C^3$, and so,  does not contain the full regular representation of $F$. More concretely, the 4 vectors $\hat \sigma(\alpha_1), ..., \hat \sigma(\alpha_4)$ of $\C^3$ are  of necessity linearly  dependent  over $\C$, so, there exists a non trivial linear relation 
$$a_1 \hat \sigma(\alpha_1)+ ... + a_4 \hat \sigma(\alpha_4) = 0 ~�{\rm in }�~ \C^3 $$
(where the complex numbers $a_i$ lie in the normal closure of $F$). Therefore, {\it any} element  $\tilde x_{\hat \sigma} =  ( x_{\sigma_1},    x_{\bar \sigma_1},  x_{\bar \sigma_2})$ of $ \Pi_{\sigma_1} \oplus \Pi_{\bar \sigma_1}   \oplus P_{\bar \sigma_2}$ satisfies :
$$(a_1 \hat \sigma^{\oplus 2}(\alpha_1)+ ... + a_4 \hat \sigma^{\oplus 2}(\alpha_4))\tilde x_{ \hat \sigma} = 0 ~�{\rm in }�~  \Pi_{\sigma_1} \oplus \Pi_{\bar \sigma_1}   \oplus P_{\bar \sigma_2}$$ 
(view each  $\hat \sigma^{\oplus 2}(\alpha_i)$  as a $(6 \times  6)$ diagonal matrice inside the $(8 \times 8)$ diagonal matrix $\rho(\alpha_i)$), i.e.  the  3 plane-components $u_{\sigma_1}, u_{\bar \sigma_1}, u_{\bar \sigma_2}$ of $u$ all vanish, and $u$ indeed lies in $P_{\sigma_2}$, 	and so in $U_A$. 

\vspace{2mm}

The existence of such a point $u =a_1 \tilde x_1 + ... + a_4 \tilde x_4$ in $U_A(K^{diff})$ establishes relation  $(\frak R)$ of \S 2.4, and concludes the proof of   Proposition 2.14.

\medskip
\noindent
Authors' addresses :

D.B. $<$daniel.bertrand@imj-prg.fr$>$

 A.P.  $<$A.Pillay@leeds.ac.uk$>$, $<$apillay@nd.edu$>$


\begin{thebibliography}{99}

\bibitem{BBP-ML} F. Benoist, E. Bouscaren, A. Pillay, On function field  Mordell-Lang and Manin-Mumford, preprint 2014
\bibitem{Bertrand-unipotent} D. Bertrand, Unipotent radicals of differential Galois groups,  Math. Ann. 321, 2001, 645-666. 
\bibitem{Bertrand-LMS} D. Bertrand, Schanuel's conjecture for nonisoconstant elliptic curves over function fields, in {\em Model Theory with Applications to Algebra and Analysis}, ed. Chatzidakis, Macpherson, Pillay, Wilkie, LMS Lecture Note Series 349, CUP, Cambridge 2008, p. 41-62. 
\bibitem{Bertrand-Fourier} D. Bertrand,  Theories de Galois diff\'erentielles et transcendance, Ann. Inst. Fourier, 59, 2009, 2773-2803.
\bibitem{Bertrand-descent} D. Bertrand, Galois descent in Galois theories,  S\'em. et Congr\`es, 23, Societ\'e Math. France, 2011, 1-24.
\bibitem{BMPZ}  D. Bertrand, D. Masser, A. Pillay, U. Zannier,  Relative Manin-Mumford for semiabelian surfaces,  arXiv:1307.1008 
\bibitem{LW} D. Bertrand and A.Pillay, A Lindemann-Weierstrass Theorem for semiabelian varieties over function fields, Journal AMS 23, 2010, 491-533.
\bibitem{Chai} C-L. Chai, A note on Manin's theorem of the kernel, American J. Math., 113 (1991), 387-389.
\bibitem{Coleman} R. Coleman, Manin's proof of the Mordell conjecture over function fields, L'Ens. Math. 36 (1990), 393-427.
\bibitem{HS} E. Hrushovski and Z. Sokolovic, Strongly minimal sets in differentially closed fields, unpublished manuscript, 1994.
\bibitem{LB} H. Lange, B. Birkenhake, {\it Complex Abelian Varieties}, Springer GMW 302, 1992.
\bibitem{DGTIII} D. Marker and A. Pillay, Differential Galois Theory III: some inverse problems. Illinois J. Math, 41 (1997), 453-461.
\bibitem{Mumford} D. Mumford,  J. Fogarty {\em Geometric Invariant Theory},  Springer 1982. 
\bibitem{DGTI}  A. Pillay, Differential Galois Theory I, Illinois J. Math, 42 (1998), 678-699. 
\bibitem{DGTII} A. Pillay, Differential Galois Theory II, Annals of Pure and Applied Logic 88 (1997), 181-191.
\bibitem{Pillay-DAG}  A. Pillay, Differential algebraic groups and the number of countable differentially closed fields, in Model Theory of Fields, Lecture Notes in Logic 5, 2006. 
\bibitem{Pillay-DGTIV}  A. Pillay, Algebraic $D$-groups and differential Galois theory, Pacific J. Math.  216 (2004), 343-360. 
\bibitem{Sh-Ta} G. Shimura, Y. Taniyama, {\it Complex Multiplication of Abelian Varieties}, Publ. Math. Soc. Japan, No 6, 1961.

\end{thebibliography}
\end{document}